\documentclass[11pt]{amsart}
\usepackage{enumerate}
\usepackage[dvips]{color}
\usepackage{epstopdf}
\usepackage{graphicx}
\graphicspath{ {images/} }
\usepackage{tikz}
\usepackage{epsfig}
\usepackage{color}
\usepackage{mathrsfs}
\DeclareMathAlphabet{\mathcalligra}{T1}{calligra}{m}{n}

\usepackage{amsmath}
\usepackage{mathtools}
\usepackage{amssymb} 
\usepackage{latexsym}
\usepackage{array}
\usepackage{ifthen}
\usepackage{amsthm}

\usepackage[dvips]{color}
\usepackage{tikz-cd}
\usepackage[all,cmtip]{xy}
\usepackage{verbatim}
\usepackage{color}

\newtheorem{theorem}{Theorem}[section]
\newtheorem{corollary}[theorem]{Corollary}
\newtheorem{lemma}[theorem]{Lemma}
\newtheorem{proposition}[theorem]{Proposition}

\theoremstyle{definition}
\newtheorem{definition}[theorem]{Definition}
\newtheorem{example}[theorem]{Example}
\newtheorem{remark}[theorem]{Remark}

\newcommand{\C}{  {\mathbb{C}}  }

\newcommand{\ep}{{\varepsilon}}

\newcommand{\si}{{\sigma}}

\newcommand{\EL}{\mathcal{L}}
\newcommand{\vf}{\varphi}

\newcommand{\Y}{\mathcal{Y}}
\newcommand{\GE}{\mathcal{G}}

\title{Equisingular Deformations of Legendrian Curves}
\author{Ana Rita Martins}
\author{Marco Silva Mendes}
\author{Orlando Neto}

\begin{document}

\maketitle

\begin{abstract}
We construct equisingular semiuniversal deformations  of Legendrian curves.
\end{abstract}

\section{Introduction}

To consider deformations of the parametrization of a Legendrian curve is 
a good first approach in order to understand Legendrian curves. 
Unfortunately, this approach cannot be generalized to higher dimensions.
On the other hand the obvious definition of deformation has its own problems.
First, not all deformations of a Legendrian curve are Legendrian. 
Second, flat deformations of the conormal of $y^k-x^n=0$ are all rigid, 
as we recall in example \ref{EX}, hence there would be too many rigid Legendrian curves.

We pursue here the approach initiated in \cite{CN}, following the Sophus Lie original approach to contact transformations: to look at [relative] contact transformations as maps that take [deformations of] plane curves into [deformations of] plane curves. We study the category of equisingular deformations of the conormal of a plane curve $Y$ replacing it by an equivalent category $\mathcal{D}\textit{ef}^{\; es,\mu}_{\,Y}$, a category of equisingular deformations of $Y$ where the isomorphisms do not come only from diffeomorphisms of the plane but also from contact transformations. Here $\mu$ stands for "microlocal", which means "locally" in the cotangent bundle (cf. \cite{KA}, \cite{KK}).

Example \ref{LIE} presents contact transformations that transform a germ of a plane curve $Y$ into the germ of a plane curve $Y^\chi$ such that $Y$ and $Y^\chi$ are not topologically equivalent or are topologically equivalent but not analytically equivalent.

We call a deformation with equisingular
plane projection an equisingular deformation of a Legendrian curve. The flatness of the plane projection is a constraint strong enough to
avoid the problems related with the use of a naive definition of deformation 
 and loose enough so that we have enough deformations.

In section \ref{SECI} we  use the results of section \ref{SECPARA} on equisingular deformations of the parametrization of a Legendrian curve
to show that there are semiuniversal equisingular deformations of a Legendrian curve. 
In particular, we show that the base space of the semiuniversal equisingular deformation is smooth.
This argument does not produce a constructive proof of the existence of the semiuniversal deformation in its standard form.
In section \ref{SECII} we construct a semiuniversal  equisingular deformation of a Legendrian curve $L$ when 
$L$ is the conormal of a Newton non-degenerate plane curve, generalizing the results of \cite{CN}.
This type of assumption was already necessary when dealing with plane curves (see \cite{GLS}). This construction is used in \cite{SQH} to extend the results of \cite{CN} and \cite{KODAIRAS}, constructing moduli spaces for Legendrian curves that are the conormal of a semiquasihomogeneous plane curve with a fixed equisingularity class.  

In section \ref{SECDEFORM} we recall some basic results on deformations of curves. In sections  \ref{S:RCG} and \ref{S:RLC} we introduce relative contact geometry (see \cite{AN}, \cite{Martins} and \cite{Mendes}).

\section{Deformations}\label{SECDEFORM}

We will only consider germs of complex spaces, maps and ideals, although sometimes we will chose convenient representatives.
We will follow the definitions and notations of \cite{GLS}.

Let $S$ be the germ of a complex space at a point $o$.
Let $\mathfrak m_S$ be the maximal ideal of the local ring $\mathcal O_{S,o}$
Let $T_oS$ be the dual of the vector space $\mathfrak m_S/\mathfrak m^2_S$. Let $X$ be a smooth manifold and $x\in X$.
We denote by $\imath$ or $\imath_S$ [$\imath_X$] the immersions $(S,o)\hookrightarrow (T_oS,0)$  
[$(X\times S,(x,o))\hookrightarrow (X\times T_oS,(x,0))$].

Let $\widetilde{\frak M}$ be an $\mathcal O_{T_oS,0}$-module 
[$\widetilde\alpha$ be a section of $\widetilde{\frak M}$, 
$\widetilde Y$ be an analytic set of $(T_oS,0)$]. 
Let ${\frak M}$ be an $\mathcal O_{S,o}$-module 
[$\alpha$ be a section of ${\frak M}$, 
$Y$ be an analytic set of $(S,o)$]. 
We say that $\widetilde{\frak M}$ 
[$\widetilde\alpha$, $\widetilde Y$] is a lifting of 
${\mathfrak M}$  
[$\alpha$, $Y$] if $\imath^*\widetilde{\mathfrak M}= \mathfrak M$ 
[$\imath^*\widetilde{\alpha}=\alpha, 
\imath^* I_{\widetilde{Y}}=I_Y$].

Let $Y$ be a reduced analytic set of $(\mathbb C^n,0)$. In order to define a deformation of $Y$ over $S$ we need to choose a section $\sigma$ of the projection $q:\mathbb C^n\times S\to S$. 
We say that a section $\widetilde\sigma: T_oS\to \mathbb C^n\times T_oS$ is a lifting of $\sigma$ if 
$\widetilde\sigma\circ i=i\circ\sigma$.
Unless we say otherwise we assume $\sigma$ to be trivial. If $S$ is reduced, $\sigma$ is trivial if and only if 
$\sigma(S)=\{0\}\times S$. In general, $\sigma$ is trivial if and only if it admits a trivial lifting to $T_oS$.

Let $\mathcal Y$ be an analytic subset of $\mathbb C^n\times S$. 
For each $s\in S$, let $\mathcal Y_s$ be the fiber of 
\begin{equation}\label{MAP}
\mathcal Y\hookrightarrow \mathbb C^n\times S \to S.
\end{equation}
 Let $i:Y\hookrightarrow \mathcal Y$ be a morphism of complex spaces that defines an isomorphism of $Y$ into $\mathcal Y_o$. We say that 
$Y\hookrightarrow \mathcal Y$ defines the \em deformation \em 
 (\ref{MAP}) of $Y$ over $S$ if 
(\ref{MAP})
 is flat.
 
 Every deformation is isomorphic to a deformation with trivial section.

Assume that $Y$ is a hypersurface of $\mathbb C^n$ and $f$ is a generator of the defining ideal of $Y$. Let $j$ be the immersion $\mathbb C^n\to \mathbb C^n\times T$ and let $r$ be the projection $ \mathbb C^n\times T\to  \mathbb C^n$. There is a generator $F$ of the defining ideal of $\mathcal Y$ such that $j^*F=f$. We say that $F$ defines a \em deformation of the equation \em of $Y$.

Let $Y\hookrightarrow \mathcal Y_i \hookrightarrow \mathbb C^n\times T\to T$ be two deformations of a reduced analytic set $Y$ over $T$. 
We say that an isomorphism $\chi:  \mathbb C^n\times T\to  \mathbb C^n\times T$ is an isomorphism of deformations if $q\circ \chi=q$, $r\circ \chi \circ j=id_{\mathbb C^n}$ and $\chi$ induces an isomorphism from $\mathcal Y_1$ onto $\mathcal Y_2$.

Given a morphism of complex spaces $f:S\to T$ and a deformation $\mathcal Y$ of $Y$ over $T$, $f^*\mathcal Y=S\times_T\mathcal Y$ defines a deformation of $Y$ over $S$.

We say that a deformation $\mathcal Y$ of $Y$ over $T$ is a \emph{versal deformation} of $Y$ if given 
\begin{itemize}
\item a closed embedding of complex space germs $f: T'' \hookrightarrow T'$,
\item a morphism $g:T'' \to T$,
\item a deformation $\mathcal Y'$ of $Y$ over $T'$ such that $f^\ast \mathcal Y' \cong g^\ast \mathcal Y$,
\end{itemize}
there is a morphism of complex analytic space germs $h: T' \to T$ such that
\[
h \circ f=g \qquad \text{and} \qquad h^\ast \mathcal Y \cong \mathcal Y'.
\]
If $\mathcal Y$ is versal and for each $\mathcal Y'$ the tangent map $T(h):T_{T'} \to T_T$ is determined by $\mathcal Y'$, $\mathcal Y$ is called a \emph{semiuniversal deformation} of $Y$.

We will now introduce deformations of a parametrization.

Assume the curve $Y$ has irreducible components $Y_1,\ldots,Y_r$. 
Set $\bar{\mathbb C}= \bigsqcup_{i=1}^r \bar{C}_i$ where each $\bar{C}_i$ is a copy of $\mathbb C$. 
Let $\varphi_i$ be a parametrization of $Y_i$, $1 \leq i \leq r$. 
The map $\varphi : \bar{\mathbb C} \to \mathbb C^n$ such that $\varphi |_{\bar{C}_i} = \varphi_i$, $1 \leq i \leq r$ is called a  \emph{parametrization} of $Y$. 

Let $\imath,\imath_n$ denote the inclusions $\bar{\mathbb C} \hookrightarrow  \bar{\mathbb C} \times T$, $\mathbb C^n \hookrightarrow  {\mathbb C}^n \times T$.
Let $\bar q$ denote the projection $\bar{\mathbb C} \times T \to T$.
We say that a morphism of complex spaces $\Phi : \bar{\mathbb C} \times T \to \mathbb C^n \times T$ is a \emph{deformation of $\varphi$ over $T$} if $\imath_n\circ\varphi=\Phi\circ \imath$ and $q_n\circ\Phi=\bar q$.

We denote by $\Phi_i$ the composition
$
\bar{C}_i \times T \hookrightarrow \bar{\mathbb C} \times T \rightarrow \mathbb C^n \times T \to \mathbb C^n$, 
$1 \leq i \leq r$.
The maps $\Phi_i$, $1 \leq i \leq r$, determine $\Phi$.
Let $\Phi$ be a deformation of $\varphi$ over $T$. Let $f: S\to T$ be a morphism of complex spaces. We denote by $f^\ast \Phi$ the deformation of $\varphi$ over $S$ given by
\[
(f^\ast \Phi)_i= \Phi_i \circ (id_{\bar{C}_i} \times f).
\]

Let $\Phi':\bar{\C} \times T \to \C^n \times T$ be another deformation of $\varphi$ over $T$. A morphism from $\Phi'$ into $\Phi$ is a pair $(\chi,\xi)$ where $\chi :\C^n \times T \to \C^n \times T$ and $\xi : \bar{\C} \times T \to \bar{\C} \times T$ are isomorphisms of complex spaces such that the diagram
\begin{equation}\label{BIGDIAGRAM}
\xymatrix{
T  &\bar{\C}\times T \ar[l] \ar[r]^{\Phi} &\C^n\times T \ar[r] &T\\
&\bar{\C} \ar@{^{(}->}[u] \ar@{_{(}->}[d] \ar[r]^{\varphi} &\C^n\times\{0\} \ar@{^{(}->}[u] \ar@{_{(}->}[d] \\
T \ar[uu]^{id_T}  &\bar{\C}\times T  \ar[l] \ar@/^2pc/[uu]^{\xi}  \ar[r]^{\Phi'} &\C^n\times T  \ar@/_2pc/[uu]_{\chi} \ar[r] &T \ar[uu]_{id_T}
}
\end{equation}
commutes.

Let $\Phi'$ be a deformation of $\varphi$ over $S$ and $f:S \to T$ a morphism of complex spaces. A \emph{morphism of $\Phi' $ into $ \Phi$ over $f$} is  a morphism from $\Phi'$ into $f^\ast \Phi$. There is a functor $p$ that associates $T$ to a deformation $\Psi$ over $T$ and $f$ to a morphism of deformations over $f$. 

Given a parametrization $\varphi$ of a plane curve $Y$ and a deformation $\Phi$ of $\varphi$, $\Phi$ is the parametrization of a hypersurface $\mathcal Y$ of $\mathbb C^2\times T$ that defines a deformation of (the equation of) $Y$.

Let $Y,Z$ be two germs of plane curves of $(\mathbb C^2,0)$.
\begin{definition}\label{EQUI}
Two plane curves $Y,Z$ are \emph{equisingular} if there are neighborhoods $V,W$ of  $0$ and an homeomorphism $\vf: V \to W$ such that $\vf(Y \cap V) =Z \cap W$.
\end{definition}
\begin{theorem}\label{TEQUI}
Let $(Y_i)_{i \in I}\, [(Z_j)_{j \in J}]$ be the set of branches $Y\,[Z]$. 
The curves $Y,Z$ are equisingular if and only if there is a bijection $\vf: I \to J$ such that $Y_i$ and $Z_{\vf(i)}$ have the same Puiseux exponents for each $i \in I$ and the contact orders $o(Y_i,Y_j)$, $o(Z_{\vf(i)}, Z_{\vf(j)})$ are equal, for each $i,j \in I$, $i \neq j$. 
\end{theorem}

The definition of \em equisingular deformation \em of the parametrization [equation] of a plane curve over a complex space is very long and technical. We will omit it. See definitions 2.36  and 2.6 of \cite{GLS}. We will present now the main properties of equisingular deformations, which characterize them completely.

\begin{theorem}\label{EQUIVALENCE}\emph{(Theorem 2.64 of \cite{GLS})}
Let $Y$ be a reduced plane curve. Let $\varphi$ be a parametrization of $Y$. Let $f$ be an equation of $Y$.
Every equisingular deformation of $\varphi$ induces a unique equisingular deformation of $f$.
Every equisingular deformation of  $f$ comes from a deformation of 
$\varphi$.
\end{theorem}

\begin{theorem}\emph{(Corollary 2.68 of \cite{GLS})}
A deformation of the equation of a reduced plane curve $Y$ over a reduced complex space is equisingular if and only if the topology of the fibers does not change.
\end{theorem}

\begin{theorem}\label{LIFTEQUI}
Let $S\hookrightarrow \mathbb (\C^k,0)$ be an immersion of complex spaces. 
Let $\varphi$ be a parametrization of a reduced plane curve.
A deformation of $\varphi$ over $S$ is equisingular if and only it admits a lifting to an equisingular deformation of $\varphi$ over $(\mathbb C^k,0)$.
\end{theorem}
\begin{proof} It follows from 
Theorem 2.38 of \cite{GLS}.
\end{proof}

\begin{proposition}\label{DECOMP}\emph{(Proposition 2.11 of \cite{GLS})}
Assume $f_1,...,f_\ell$ define germs of reduced irreducible curves of $(\mathbb C^2,0)$ and $F$ defines an equisingular deformation over a germ of complex space $S$ of the curve defined by $f_1\cdots f_\ell$. Then $F=F_1\cdots F_\ell$, where each $F_i$ defines an equisingular deformation of $f_i$ over $S$.
\end{proposition}

\section{Relative contact geometry}\label{S:RCG}
We usually identify a subset of $\mathbb P^{n-1}$ with a conic subset of $\C^n$. Given a manifold $M$ we will also identify a subset of the projective cotangent bundle $\mathbb P^\ast M$ with a conic subset of the cotangent bundle $T^\ast M$ (for the canonical $\C^\ast$-action of $T^\ast M$).

Let $q:X\to S$ be a morphism of complex spaces.
Let $p_i$, $i=1,2$ be the canonical projections from $X\times_SX$ to $X$.
Let $\Delta$ denote the diagonal of $X\to X\times_SX$ and the diagonal immersion $X\hookrightarrow X\times_SX$.
Let $I_\Delta$ be the defining ideal of the diagonal of $X\times_SX$.
We say that the coherent $\mathcal O_X$-module 
$\Omega^1_{X/S}=\Delta^*(I_\Delta/I_\Delta^2)$ is the sheaf of \em relative differential forms of $X\to S$ \em (see \cite{HA}).

Given a local section $f$ of $\mathcal O_X$ set $f_i=f\circ p_i$, $i=1,2$. Consider the morphism $d:\mathcal O_X\to \Omega^1_{X/S}$ given by
$$
f\mapsto f_1-f_2 \qquad \hbox{\rm mod } I^2_\Delta.
$$
Notice that, given an open set $U$ of $X$ and $f,g\in\mathcal O_X(U)$,
$\varphi\in q^{-1}\mathcal O_S$,
\begin{equation}\label{DERIVATIONCAUSE}
d(fg)=fdg+gdf, \qquad
\hbox{\rm and} \qquad
d(\varphi f)=\varphi df
\end{equation}
If $x_1,...,x_n\in\mathcal O_X(U)$ are such that
$\Omega^1_{X/S}|_U\xrightarrow{\sim}\oplus_{i=1}^n\mathcal O_{U}dx_i$,
we say that $(x_1,...,x_n)$ is a \em partial system of local coordinates on $U$ of \em $X\to S$.

Notice that $(x_1,...,x_n)$ is a partial system of local coordinates  of  $X\to S$ on $U$ if and only if 
$\Omega^n_{X/S}|_U=\mathcal O_Udx_1\wedge\cdots\wedge dx_n$.

If $(x^1,...,x^n)$ is a partial system of local coordinates on $U$ of  $X\to S$, $x^i_1-x^i_2$, $i=1,...,n$, generate $I_\Delta|_U$.
Given $f\in \mathcal O_X(U)$, there are $a_i\in\mathcal O_X(U)$ such that $df=\sum_{i=1}^na_idx^i$. We set
$$
\frac{\partial f}{\partial x_i}=a_i,\qquad i=1,...,n.
$$
When $M,S$ are manifolds, $X=M\times S$ and $q$ is the projection $M\times S\to S$ this definition of partial derivative coincides with the usual one because of (\ref{DERIVATIONCAUSE}).
When $S$ is a point, $\Omega^1_{X/S}$ equals the sheaf of differential forms $\Omega^1_{X}$.

If $\Omega^1_{X/S}$ is a locally free $\mathcal O_X$-module, we denote by $\pi=\pi_{X/S}:T^*(X/S)\to X$ the vector bundle with sheaf of sections $\Omega^1_{X/S}$. Whenever it is reasonable we will write $\pi$ instead of $\pi_{X/S}$. We denote by $\tau_{X/S}:T(X/S)\to X$ the dual vector bundle of $T^*(X/S)$. We say that $T(X/S)$ [$T^*(X/S)$] is the \em relative tangent bundle \em [\em cotangent bundle\/\em ] of $X\to S$.

Let $\varphi:X_1\to X_2$, $q_i:X_i\to S$ be morphisms of complex spaces such that $q_2\varphi=q_1$. Let $\Delta_i: X_i\to X_i\times_S X_i$ be the diagonal map, $i=1,2$. If we denote by $\varphi_S$ the canonical map from $X_1\times_S X_1$ to $X_2\times_S X_2$, $\varphi_S^*:I_{\Delta_2}\to I_{\Delta_1}$ induces a morphism 
$\varphi^*: \Omega^1_{X_2/S}\to \Omega^1_{X_1/S}$ that generalizes the
pullback of differential forms. Moreover, $\varphi^*$ induces a morphism of $\mathcal O_{X_1}$-modules
\begin{equation}\label{ROPHI1}
\widehat\rho_\varphi:\varphi^*\Omega^1_{X_2/S}
=\mathcal O_{X_1}\otimes _{\varphi^{-1}\mathcal O_{X_2}}
\varphi^{-1}\Omega^1_{X_2/S} \to
\Omega^1_{X_1/S}.
\end{equation}
If $\Omega^1_{X_i/S}$, $i=1,2$, and the kernel and cokernel of (\ref{ROPHI1}) are locally free, we have a morphism of vector bundles
\begin{equation}\label{ROPHI2}
\rho_\varphi: X_1\times_{X_2}T^*(X_2/S) \to T^*(X_1/S).
\end{equation}
If $\varphi$ is an inclusion map, we say that the kernel of (\ref{ROPHI2}), and its projectivization, are the \em conormal bundle of $X_1$ relative to $S$. \em
We will denote by $T^*_{X_1}(X_2/S)$ or $\mathbb P^*_{X_1}(X_2/S)$ the conormal bundle of $X_1$ relative to $S$.
We denote by
$$
\varpi_\varphi:T(X_1/S) \to X_1\times_{X_2} T(X_2/S)
$$
the dual morphism of $\rho_\varphi$. We say that $\varpi_\varphi$ is the \em relative tangent morphism of $\varphi$ over $S$. \em
These are straightforward generalizations of the constructions of \cite{KA}.

If $(x_1,...,x_n)$ is a partial system of local coordinates of $X\to S$ and $(y_1,...,y_m)$ is a system of local coordinates of a manifold $Y$, $(x_1,...,x_n,y_1,...,y_m)$ is a partial system of local coordinates of $X\times Y\to X\to S$. Hence $\Omega^1_{X/S}$ locally free implies $\Omega^1_{X\times Y/S}$ locally free.
Moreover, if $\Omega^1_{X/S}$ is locally free and $E\to X$ is a vector bundle, $\Omega^1_{E/S}$ is locally free.

Let $(x_1,...,x_n)$ be a partial system of local coordinates of $X\to S$ on an open set $U$ of $X$. Set $V=\pi_{X/S}^{-1}(U)$. There are $\xi_1,...,\xi_n\in\mathcal O_{T^*(X/S)}(V)$ such that, for each $\sigma\in V$, 
$$
\sigma=\textstyle{\sum_{i=1}^n}\xi_i(\sigma)dx_i.$$
Notice that $(x_1,...,x_n,\xi_1,...,\xi_n)$ is a partial system of local coordinates of $T^*(X/S)\to S$.
Let $o\in X$, $u\in T_\sigma T^*(X/S)$. Let
$$
\varpi_\pi (\sigma): T_\sigma (T^*(X/S)/S)\to T_o(X/S) 
$$ 
be the relative tangent morphism of $\pi$ over $S$ at $\sigma$.
There is one and only one $\theta\in\Omega^1_{T^*(X/S)/S}$ such that,
$$
\theta(\sigma)(u)=\sigma(\varpi_\pi(\sigma)(u)),
$$
for each $o\in X$, each $\sigma\in T^*_o(X/S)$ and each 
$u\in T_\sigma(T^*(X/S)/S)$. Given a partial system of local coordinates $(x_1,...,x_n)$ of $X\to S$ on an open set $U$,
$$
\theta|_{\pi^{-1}(U)}=\textstyle{\sum_{i=1}^n}\xi_idx_i.
$$
We say that $\theta_{X/S}=\theta$ is the \em canonical \em $1$-form of $T^*(X/S)$.

Notice that $(d\theta)(\sigma)$ is a symplectic form of 
$T_\sigma(T^*(X/S)/S)$, for each $\sigma\in T^*(X/S)$. 
We say that $(x_1,...,x_n,\xi_1,...,\xi_n)$ is a \em partial system of symplectic coordinates of \em $T^*(X/S)$ (associated to $(x_1,...,x_n)$).

Assume $M$ is a manifold.
When $q$ is the projection $M\times S\to S$ we will replace "$M\times S/S$" by $"M|S"$. Let $r$ be the projection $M\times S \to M$.
Notice that $\Omega^1_{M|S}\xrightarrow{\sim}\mathcal O_{M\times S}
\otimes_{r^{-1}\mathcal O_M}r^{-1}\Omega^1_M$
is a locally free $\mathcal O_{M\times S}$-module.
Moreover, $T^*(M|S)=T^*M\times_M(M\times S)$. If $\imath$ is the inclusion $T^*(M|S)\hookrightarrow T^*(M\times S)$, 
$\imath^*\theta_{M\times S}=\theta_{M|S}$.
A system of local coordinates of $M$ is a partial system of local coordinates of $M\times S\to S$.

We say that $\Omega^1_{M|S}$ is the \em sheaf of relative differential forms of $M$ over $S$. \em
We say that $T^*(M|S)$ is the \em relative cotangent  bundle of $M$ over $S$. \em

Let $N$ be a complex manifold of dimension $2n-1$. 
Let $S$ be a complex space.
We say that a section $\omega$ of $\Omega^1_{N|S}$ is a \em relative contact form of \em $N$ over $S$ if  $\omega\wedge d\omega^{n-1}$ is a local generator of 
$\Omega^{2n-1}_{N|S}$.
Let $\mathfrak C$ be a locally free subsheaf of $\Omega^1_{N|S}$. 
We say that $\mathfrak C$ is a \em structure of relative contact manifold on \em $N$ over $S$ if $\mathfrak C$ is locally generated by a relative contact form of $N$ over $S$.  We say that $(N\times S,\mathfrak C)$ is a \em relative contact manifold over $S$. \em
When $S$ is a point we obtain the usual notion of contact manifold.

Let $(N_1\times S,\mathfrak C_1)$, $(N_2\times S,\mathfrak C_2)$ be relative contact manifolds over $S$. 
Let $\chi$ be a morphism from $N_1\times S$ into $N_2\times S$ such that 
$q_{N_2}\circ \chi =q_{N_1}$. 
We say that $\chi$ is a \em relative contact transformation \em of $(N_1\times S,\mathfrak C_1)$ into $(N_2\times S,\mathfrak C_2)$ if the pull-back by $\chi$ of each local generator of $\mathfrak C_2$ is a local generator of $\mathfrak C_1$.

We say that the projectivization $\pi_{X/S}:\mathbb P^*(X/S)\to X$ of the vector bundle 
$T^*(X/S)$ is the \em projective cotangent bundle \em of $X \to S$.

Let $(x_1,...,x_n)$ be a partial system of local coordinates on an open set $U$ of $X$. Let $(x_1,...,x_n,\xi_1,...,\xi_n)$ be the associated partial system of symplectic coordinates of $T^*(X/S)$ on $V=\pi^{-1}(U)$.
Set $p_{i,j}=\xi_i\xi_j^{-1}$, $i\not=j$,
$$
V_i=\{ (x,\xi)\in ~V: ~\xi_i\not =0  \}, 
\qquad \omega_i=\xi_i^{-1}\theta,
\qquad i=1,...,n.
$$
each $\omega_i$ defines a relative contact form $dx_j-\sum_{i\not=j}p_{i,j}dx_i$ on $\mathbb P^*(X/S)$, endowing  $\mathbb P^*(X/S)$ with a structure of relative contact manifold over $S$.

Let $\omega$ be a germ at $(x,o)$ of a relative contact form of $\mathfrak C$.
A lifting $\widetilde\omega$ of $\omega$ defines a germ 
$\widetilde{\mathfrak C}$ of a relative contact structure of 
$N\times T_oS \to T_oS$. Moreover, $\widetilde{\mathfrak C}$ is a lifting of 
the germ at $o$ of ${\mathfrak C}$.

Let $(N\times S,\mathfrak C)$ be a relative contact manifold over a complex manifold $S$.  
Assume $N$ has dimension $2n-1$ and $S$ has dimension $\ell$.
Let $\mathcal L$ be a reduced analytic set of $N\times S$ of pure dimension $n+\ell-1$.
We say that $\mathcal L$ is a \em relative Legendrian variety \em of $N\times S$ over $S$ if for each section 
$\omega$ of $\mathfrak C$, $\omega$ vanishes on the regular part of $\mathcal L$. When $S$ is a point, we say that $\mathcal L$ is a \em  Legendrian variety \em of $N$.

Let $\mathcal L$ be an analytic set of $N\times S$. Let $(x,o) \in \mathcal L$. Assume $S$ is an irreducible germ of a complex space at $o$.
We say that  $\mathcal L$ is a \em relative Legendrian variety of $N$ over $S$  at \em $(x,o)$  if there is a relative Legendrian variety 
$\widetilde{\mathcal L}$ of $(N,x)$ over $(T_oS,0)$  that is a lifting of 
the germ of $\mathcal L$ at $(x,o)$. Assume $S$ is a germ of a complex space at $o$ with irreducible components $S_i, i\in I$. We say that  $\mathcal L$ is a \em relative Legendrian variety of $N$ over $S$  at \em $(x,o)$  if $S_i \times_S \mathcal L$ is a relative Legendrian variety of $S_i\times_S N$ over $S_i$ at $(x,o)$, for each $i \in I$.

We say that  $\mathcal L$ is a \em relative Legendrian variety 
 of $N\times S$ \em if $\mathcal L$ is a relative Legendrian variety of $N\times S$ at $(x,o)$ for each $(x,o)\in \mathcal L$.

The main problem of defining relative Legendrian variety over a complex space $S$ comes from the fact that $S$ does not have to be pure dimensional, hence we cannot assign a pure dimension to the Legendrian variety.

\begin{lemma}
Let $\chi$ be a relative contact transformation from $(N_1\times S, \mathfrak C_1)$ into $(N_2\times S, \mathfrak C_2)$.
Let $\mathcal L_1$ be a relative Legendrian curve of $(N_1\times S, \mathfrak C_1)$.
If $\mathcal L_2$ is the analytic subset of $N_2\times S$ defined by the pull back by $\chi^{-1}$ of the defining ideal of $\mathcal L_1$, 
$\mathcal L_2$ is a relative Legendrian variety of $(N_2\times S, \mathfrak C_2)$.
\end{lemma}

\begin{proof}
Let $\chi:(N_1\times S,\mathfrak C_1)\to(N_2\times S,\mathfrak C_2)$
be a relative contact transformation over $S$.
Let $(x_1,o)$ be a point of $N_1\times S$. Set $(x_2,o)=\chi(x_1,o)$.
There is a morphism of germs of complex spaces
$$
\widetilde\chi:(N_1\times T_oS,(x_1,o))
\to
(N_2\times T_oS,{(x_2,o)})
$$
such that $\widetilde\chi\circ\imath_{N_1}=\imath_{N_2}\circ\chi$.
We say that such a morphism is a \em lifting \em of $\chi$.
Let $\widetilde{\mathfrak{C}}_2$ be a lifting of ${\mathfrak{C}}_2$ at $(x_2,o)$.
Then $\widetilde{\mathfrak{C}}_1=\widetilde\chi^*\widetilde{\mathfrak{C}}_2$ is a lifting of $\mathfrak C_1$ at $(x_1,o)$. Moreover, $\widetilde\chi$ is a germ of a relative contact transformation.

Let $\mathcal L_1$ be a germ of a relative Legendrian variety at $(x_1,o)$.
There is a lifting $\widetilde{\mathcal L}_1$ of $\mathcal L_1$ that is a germ of relative Legendrian variety of $N_1\times T_oS$.
Hence $\widetilde\chi(\widetilde{\mathcal L}_1)$ is a germ of a relative Legendrian variety of $N_2\times T_oS$ and 
$\widetilde\chi(\widetilde{\mathcal L}_1)$ is a lifting of $\mathcal L_2$ at $(x_2,o)$.
\end{proof}

Let $Y$ be a reduced analytic set of $M$.
Let $\mathcal Y$ be a flat deformation of $Y$ over $S$.
Set $X=M\times S\setminus \mathcal Y_{\rm sing}$.
We say that the Zariski closure of $\mathbb P^*_{\mathcal Y_{\rm reg}}(X/S)$ in $\mathbb P^*(M|S)$ is the \em conormal
$\mathbb P^*_{\mathcal Y}(M|S)$ of $\mathcal Y$ over $S$. \em

\begin{theorem}
The conormal of $\mathcal Y$ over $S$ is a relative Legendrian variety of $\mathbb P^*(M|S)$.
If $\mathcal Y$ has irreducible components $\mathcal Y_1,...,\mathcal Y_r$,
\begin{equation}\label{DECOMPOSITION}
\mathbb P^*_\mathcal{Y}(M|S)=
\cup_{i=1}^r
\mathbb P^*_{\mathcal{Y}_i}(M|S).
\end{equation}
\end{theorem}

\begin{proof} We have a commutative diagram
\[
  \begin{tikzcd}
& \mathcal{Y}_{\rm reg}  \arrow[r,hook,"i"] \arrow[d, "\Delta_{\mathcal Y_{\rm reg}}"]
& X \arrow[d, "\Delta_{X}"] \\
& \mathcal{Y}_{\rm reg}\times_S\mathcal{Y}_{\rm reg} \arrow[r, hook,"j"]
& X\times_S X \\  
  \end{tikzcd}
\]
Since $I_{\Delta_{\mathcal Y_{\rm reg}}}=
j^{\ast}\left( (I_{\Delta_X}+I_{ \mathcal Y_{\rm reg}\times_S \mathcal Y_{\rm reg} })/
I_{ \mathcal Y_{\rm reg}\times_S \mathcal Y_{\rm reg} }\right)$,
\begin{equation}\label{DIAGONALI}
\Delta_{\mathcal Y_{\rm reg}}^{\ast}
(I_{\Delta_{\mathcal Y_{\rm reg}}}
/I^2_{\Delta_{\mathcal Y_{\rm reg}}})
~
\xrightarrow{\sim}
~
i^{\ast}\Delta_X^{\ast}(
(I_{\Delta_X}+
I_{\mathcal Y_{\rm reg}\times_S\mathcal Y_{\rm reg} })/(I^2_{\Delta_X}+I_{\mathcal Y_{\rm reg}\times_S
\mathcal Y_{\rm reg}})).
\end{equation}
Let $(x,o)\in Y_{\rm reg}$. Let $\widetilde{\mathfrak m}$ 
be the ideal of $\mathcal O_{M\times S,(x,o)}$
generated by $\mathfrak{m}_o$.
Let $(y_1,...,y_n)$ be a system of local coordinates of 
$(M,x)$ such that $I_{Y,x}=(y_{k+1},...,y_n)$.
There are $F_j\in \mathcal O_{M\times S, (x,o)}$, $j=k+1,...,n$
such that $I_{\mathcal Y,(x,o)}=(F_{k+1},...,F_n)$ and 
$F_j-y_j\in \widetilde{\mathfrak m}$, $j=k+1,...,n$.
Set
$$
x^i=y_i, \qquad i=1,...,k, \qquad   x^i=F_i,  \qquad i=k+1,...,n.
$$
Notice that $(x^1,...,x^n)$ is a partial system of local coordinates of $X\to S$. Since near $(x,o)$
$$
I_{\Delta_X}=(x^1_1-x^1_2,...,x^n_1-x^n_2)
~~~\hbox{ \rm and } ~~~
I_{{\mathcal Y\times_s\mathcal Y}}
=(x_1^{k+1},..., x_1^n,x_2^{k+1},...,x^n_2),
$$
it follows from (\ref{DIAGONALI}) that
$dx^1,...,dx^k$ is a local basis of $\Omega^1_{\mathcal Y/S}$,
$dx^1,...,dx^n$ is a local basis of $\Omega^1_{M|S}$,
$$
\widehat\rho_i(dx^j)=dx^j,
\quad
j=1,...,k,
\qquad
\hbox{\rm and}
\qquad
\widehat\rho_i(dx^j)=0,
\quad
j=k+1,...,n.
$$
Hence the kernel of $\widehat\rho_i$ at $(x,o)$ equals 
$\oplus_{j=k+1}^n\mathbb C\{x^1,...,x^k  \}dx^j$.
Given the partial system of symplectic coordinates
$(x^1,...,x^n, \xi^1,...,\xi^n)$,
the ideal of the kernel of
$$
\rho_i:\mathcal Y_{\rm reg}\times_XT^*(X/S)\to T^*(\mathcal Y_{\rm reg}/S)
$$
is generated by $x^{k+1},...,x^n,\xi^1,...,\xi^k$.

It is enough to prove the second statement when $S$ is smooth.
Its proof relies on the fact that each connected component of 
$\mathcal Y$ is dense in one of the irreducible components of $\mathcal Y$.
\end{proof}

Let $q:X\to S$ be a morphism of complex spaces.
Let $y\in Y\subset X$.
We say that $Y$ is a \em submanifold of $X\to S$ at $y$ \em if there is a partial system of local coordinates $(x_1,...,x_n)$ of $X\to S$ near $y$ and $1\le k\le n$ such that $Y=\{x_1=\cdots=x_k=0\}$ near $y$.
We say that $Y$ is a \em submanifold \em of $X\to S$ if $Y$ is a  submanifold of $X\to S$ at $y$ for each $y\in Y$. 

Notice that a submanifold of $X\to S$ is not necessarily a manifold, although it behaves like one in several ways.

Let $Y\subset X$. 
Let 
$\gamma:\Delta_\varepsilon=\{t\in \mathbb C: |t|<\varepsilon \}\to Y$
be a holomorphic curve such that $\gamma(0)=y$.
We associate to $\gamma$ a tangent vector $u$ of $Y$ at $y$ setting
$u\cdot f=(f\circ \gamma)'(0))$, for each $f\in\mathcal O_{X,y}$.
We associate to $\gamma$ an element $u$ of $T_y(X/S)$ 
setting
\begin{equation}\label{RELATIVETANGENT}
u\cdot f=df(y)(\gamma'(0)), \qquad f\in\mathcal O_{X,y}.
\end{equation}
If $Y$ is a submanifold of $X\to S$ the set of relative vector fields (\ref{RELATIVETANGENT}) define a linear subspace $T_y(Y/S)$ of 
$T_y(X/S)$.

Let us fix a point $o$ of $S$.
Consider the canonical maps
\[
T^*M \xrightarrow{i} T^*(M|S)=(T^*M)\times S \xrightarrow{r} T^*M.
\]
Since $T_\sigma(T^*(M|S)/S)=T_{r(\sigma)}T^*M$ and
\[
(d\theta_{M|S})(\sigma)=(i^*d\theta_M)(r(\sigma)),
\]
$(d\theta_{M|S})(\sigma)$ is a symplectic form of $T_\sigma(T^*(M|S)/S)$.

The Poisson bracket of $(T^*M)$ induces a Poisson bracket in $T^*(M|S)$. Let $f \in \mathcal O_{T^*(M|S)}$. Setting $f_s(x,\xi)=f(x,\xi,s)$
\[
\{f,g\}_{T^*(M|S)}(x,\xi,s)=\{f_s,g_s\}_{T^\ast M}(x,\xi).
\]
Let $W$ be a submanifold of $T^\ast(M|S)$. Setting $W_s=\{(x,\xi) \in T^\ast M: (x,\xi,s) \in W\}$, $W$ is an involutive submanifold of $T^\ast(M|S)$ if and only if $W_s$ is an involutive submanifold of $T^\ast M$ for each $s \in S$. It is well known that $W_s$ is an involutive submanifold of $T^\ast M$ if and only if $T_\sigma W_s$ is an involutive linear subspace of $T_\sigma T^\ast M$ for each $\sigma \in W_s$

\begin{lemma}\label{L:lagrangian}
Let $\mathcal L$ be a conic submanifold of  $T^\ast(M|S)$. The manifold $\mathcal L$ is a Legendrian submanifold of $\mathbb P^\ast(M|S)$ if and only if $T_\sigma(\mathcal L/S)$ is a linear Lagrangian subspace of $T_\sigma(T^\ast(M|S)/S)$ for each $\sigma \in \mathcal L$.
\end{lemma}
\begin{proof}
The submanifold $W$ considered above is an involutive submanifold of $T^\ast(M|S)$ if and only if $T_\sigma (W/S)$ is a linear involutive subspace of $T_\sigma(T^\ast(M|S)/S)$ for each $\sigma \in W$. The result follows from an argument of dimension.
\end{proof}

\begin{theorem}\label{T:RCONORMAL}
Let $\mathcal L$ be an irreducible germ of a relative Legendrian analytic set of $\mathbb P^*(M|S)$. 
If the analytic set $\pi(\mathcal L)$ is a flat deformation over $S$ of an analytic set of $M$, $\mathcal L=\mathbb P^*_{\pi(\mathcal L)}(M|S)$.
\end{theorem}

\begin{proof}
There is $s \in S$ such that $Y\times\{s\} \subset \mathcal Y$. Let $o$ be a smooth point of $Y$. There is an open set $U$ of $Y$ and a system of local coordinates $(y_1,\ldots,y_n)$ on $U$ such that $Y\cap U=\{y_1=\cdots=y_k=0\}$. Since $Y$ is flat, there is a neighborhood $V$ of $s$ and a system of partial symplectic coordinates $(x_1,\ldots,x_n,\xi_1,\ldots,\xi_n)$ on $\pi^{-1}(U \times V)$ such that
$$
\pi(\mathcal L)\cap U\times V=\{x_1=\cdots=x_k=0\}.
$$
Repeating the argument of Lemma \ref{L:lagrangian},
$$
{\mathcal L} \cap \pi^{-1}(\pi(\mathcal L)_{reg})=\mathbb P_{\mathcal Y_{reg}}^\ast (M\times S \setminus \mathcal Y_{sing}/S).
$$
Since $\mathcal L$ is closed $\mathbb P_{\mathcal Y}^\ast (M|S) \subseteq \mathcal L$.
Since $\mathcal L$ is irreducible and both spaces have the same dimension, the inclusion is an equality.
\end{proof}

We present now an alternative construction of the conormal of a flat deformation of a hypersurface. This construction is more suitable to
compute the conormal using computer algebra methods. For this purpose it is enough to consider the case where $S$ is smooth.

Let $F$ be a generator of the defining ideal of $\mathcal Y$.
Let $\mathfrak J_{F,(x_i)}$ be the ideal of $\mathbb C\{c,x,\xi,s\}$ generated by 
\begin{equation}\label{generatorsxi}
F,~ \xi_i-cF_{x_i}, \qquad i=1,...,n.
\end{equation} 

The ideal
$$
\mathfrak K_{F,(x_i)}=
\mathfrak J_{F,(x_i)}\cap 
\mathbb C\{x,\xi,s\}.
$$
defines a conic analytic subset of $T^*M\times S$, hence it also defines an analytic subset $\mathcal Con_S\mathcal{Y}$ of $\mathbb P^*(M|S)$. 

\begin{lemma}
The ideal $\mathfrak K_{F,(x_i)}$ does not depend on the choice of $F$ or $(x_i)$.
\end{lemma}

\begin{proof}
Given another system of local coordinates $(y_i)$ there are function $\eta_i$ such that
$
\sum_i\eta_i dy_i=\sum_i\xi_i dx_i.
$
Since
$$
\textstyle{\sum_i}\eta_idy_i=\textstyle{\sum_i}\eta_i\Sigma_j\textstyle\frac{\partial y_i}{\partial x_j}dx_j=
\Sigma_j\textstyle{\sum_i}\textstyle\frac{\partial y_i}{\partial x_j}\eta_idx_j,
$$
$$
\xi_j- c F_{x_j}
=\textstyle{\sum_i}\textstyle\frac{\partial y_i}{\partial x_j}\eta_i-
c\textstyle{\sum_i} F_{y_i}\textstyle\frac{\partial y_i}{\partial x_j}
=\textstyle{\sum_i}\textstyle\frac{\partial y_i}{\partial x_j}(\eta_i-c  F_{y_i}).
$$
Since the Jacobian matrix of the coordinate change is invertible, $\mathfrak J_{F,(x_i)}$ does not depend on $(x_i)$.

Assume that $\varphi$ does not vanish. 
Since $\xi_i-c  (\varphi F)_{x_i}=
\xi_i-c\varphi   F_{x_i}- cF \varphi_{x_i}$,
$\mathfrak J_{\varphi F}$ is generated by 
\begin{equation}\label{generatorsprime}
F,~ \xi'_i-c F_{x_i}, \qquad i=1,...,n,
\end{equation}
where $\xi'_i=\varphi^{-1}\xi_i$, $i=1,...,n$. 

Consider the actions of $\mathbb C^*$ into $T^*M\times S\times \mathbb C$ and $T^*M\times S$ given by
$$
t\cdot ((x_i),(\xi_i),(s_j),c)=((x_i),(t\xi_i),(s_j),tc),
$$
$$
t\cdot ((x_i),(\xi_i),(s_j))=((x_i),(t\xi_i),(s_j)).
$$
By (\ref{generatorsxi}), the ideals $\mathfrak J_F$  [$\mathfrak K_F$] are generated by homogeneous polynomials on $\xi_1,...,\xi_n,c$ [$\xi_1,...,\xi_n$].
Assume that $\mathfrak K_F$ is generated by the homogeneous polynomials 
$$
P_k(\xi_1,...,\xi_n), k=1,...,m.
$$
It follows from (\ref{generatorsxi}) and (\ref{generatorsprime}) that 
$\mathfrak K_{\varphi F}$ is generated by  $P_k(\xi'_1,...,\xi'_n), k=1,...,m$.
If $P_k$ is homogeneous of degree $d_k$, $P_k(\xi'_1,...,\xi'_n)=\varphi^{-d_k}P_k(\xi_1,...,\xi_n)$.
Hence $\mathfrak K_F=\mathfrak K_{\varphi F}$.
\end{proof}

\begin{theorem}
If $\mathcal Y$ is a flat deformation over $S$ of a hypersurface of $M$, $\mathbb P^*_\mathcal{Y}(M|S)=\mathcal Con_S\mathcal Y$.
\end{theorem}
\begin{proof}
If $\mathcal Y$ is non singular at a point $o$, there is a partial system of symplectic coordinates $(x_1,...,x_n,\xi_1,...,\xi_n)$ such that $F=x_1$ in a neighborhood $U$ of $o$.
Hence $\mathfrak J_{F,(x_i)}$ is generated by 
\begin{equation}\label{preconormal}
\xi_1-c,\xi_2,...,\xi_n,x_1.
\end{equation}
Therefore $\mathfrak K_{F,(x_i)}$ is generated by $x_1,\xi_2,...,\xi_n$. 
Hence $\mathbb P^*_\mathcal{Y}(M|S)=\mathcal Con_S\mathcal Y$ in $\pi^{-1}(U)$. 
Therefore $\mathcal Con_S\mathcal Y$ contains $\mathbb P^*_\mathcal{Y}(M|S)$. 
Assume that there is an irreducible component $\Gamma$ of $\mathcal Con_S\mathcal Y$ that is not contained in $\mathbb P^*_\mathcal{Y}(M|S)$.
 Then $\Gamma$ is contained in $\mathcal Y_{\rm sing}\times_{M\times S}\mathbb P^*(M\times S|S)$.
Hence the set of zeros of $\mathfrak J_{f,(x_i)}$ contains points of 
$$
\mathcal Y_{\rm sing}\times_{M\times S}T^*M\times S\times \mathbb C\setminus M\times S\times \mathbb C.
$$
But it follows from (\ref{generatorsxi}) that the intersection of the set of zeros of $\mathfrak J_{F,(x_i)}$ with $\mathcal Y_{\rm sing}\times_{M\times S}T^*M\times S\times \mathbb C$ is contained in $M\times S\times \mathbb C$.
\end{proof}

The following Singular routine (see \cite{DGPS}) computes the relative conormal of the hypersurface defined by $z^2+y^3+sx^4$ when we assume $\theta=udx+vdy+wdz$ and we look at $s$ as a deformation parameter.

\medskip
{\tt
\noindent
ring r=0,(c,u,v,w,x,y,z,s),dp;\\
poly F=z2+y3+sx4;\\
ideal I=F,u-c*diff(F,x),v-c*diff(F,y),w-c*diff(F,z);\\
ideal J=eliminate(I,c);\\
J;
}

\medskip\noindent
If we consider the suitable contact coordinates and  choose a different ordering we can reduce substantially the number of equations we obtain. 
\medskip

Let $T_\varepsilon$ be the complex space with local ring $\mathbb C\{\varepsilon\}/(\varepsilon^2)$.
Let $I,J$ be ideals of the ring $\mathbb C\{s_1,...,s_m\}$. Assume $J\subset I$. 
Let $X,S,T$ be the germs of complex spaces with local rings $\mathbb C\{x,y,p\}$,
$\mathbb C\{s\}/I,\mathbb C\{s\}/J$. Consider the maps
$i:X\hookrightarrow X\times S$, $j: X\times S \hookrightarrow  X\times T$ and $q: X\times S \to S$.

Let $\frak m_X,\frak m_S$ be the maximal ideals of 
$\mathbb C \{x,y,p\}$, $\mathbb C \{s\}/I$.
Let $\frak n_S$ be the ideal of $\mathcal O_{X\times S}$
generated by $\frak m_X\frak m_S$.

Let $\chi: X\times S \to X\times S$ be a relative contact transformation.
If $\chi$ verifies 
\begin{equation}\label{DIDENTITY}
\chi \circ i=i,  ~~ q \circ \chi=q ~ \hbox{ \rm and } ~ \chi(0,s)= (0,s) ~~
\hbox{ \rm for each } s.
\end{equation}
 there are 
$\alpha,\beta,\gamma\in \frak n_S$ such that
\begin{equation}\label{abc}
\chi(x,y,p,s)=(x+\alpha,y+\beta,p+\gamma,s).
\end{equation}

\begin{theorem}\label{T:RCKT}
\begin{enumerate}[$(a)$]

\item Let $\chi : X\times S \to X \times S$ be a relative contact transformation that verifies \em (\ref{DIDENTITY}). \em
Then $\gamma$ is determined by $\alpha$ and $\beta$. Moreover, there is $\beta_0 \in \mathfrak{n}_S+p\mathcal O_{X\times S}$ such that $\beta$ is the solution of the Cauchy problem
\begin{equation}\label{E:CAUCHY}
\left( 1+\frac{\partial \alpha}{\partial x}  + p\frac{\partial \alpha}{\partial y}\right)\frac{\partial \beta}{\partial p} - p\frac{\partial \alpha}{\partial p}\frac{\partial \beta}{\partial y} - \frac{\partial \alpha}{\partial p}\frac{\partial \beta}{\partial x}=p\frac{\partial \alpha}{\partial p},  
\end{equation}
$\beta+p\mathcal O_{X\times S} = \beta_0$.

\item Given $\alpha\in \mathfrak{n}_S$, $\beta_0 \in \mathfrak{n}_S+p\mathcal O_{X\times S}$, 
there is a unique relative contact transformation  $\chi$ that verifies \em (\ref{DIDENTITY}) \em
and the conditions of statement $(a)$. 
We denote $\chi$ by $\chi_{\alpha,\beta_0}$.
\item 
If  $S=T_\varepsilon$ the Cauchy problem \emph{(\ref{E:CAUCHY})}  simplifies into 
\begin{equation}\label{E:CAUCHYTEBETA}
\frac{\partial \beta}{\partial p}=p\frac{\partial \alpha}{\partial p}, \qquad \beta +  p\mathcal O_{X\times T_\varepsilon} = \beta_0.
\end{equation}
\item \em
Let  $\chi=\chi_{\alpha,\beta_0}:X\times T \to X\times T$ be a relative contact transformation. Then, $\chi$ is a lifting to $T$ of $j^*\chi=\chi_{j^*\alpha,j^*\beta_0}:X\times S \to X\times S$. \em
If $\chi$ equals \emph{(\ref{abc})}, 
\[
j^*\chi(x,y,p,s)=(x+j^*\alpha,y+j^*\beta,p+j^*\gamma,s).
\]

\item
Assume $\mathcal O_T=\mathbb C\{s\}$, $\mathcal O_{T_0}= 
\mathbb C\{s,\varepsilon\}/(\varepsilon^2,\varepsilon s_1,\ldots \varepsilon s_m)$.
Given a relative contact transformation 
\begin{equation}\label{CHICHI}
\chi(x,y,p,s)=(x+A,y+B,p+C,s)
\end{equation}
over $T$ and $\alpha,\beta, \gamma \in \mathfrak{m}_X$, 
\begin{equation}\label{AEPAL}
\chi_0(x,y,p,s,\varepsilon)=
(x+A+\varepsilon\alpha,y+B+\varepsilon\beta,p+C+\varepsilon\gamma,s,\varepsilon)
\end{equation}
is a relative contact transformation over $T_0$  if and only if 
\begin{equation}\label{TEP}
(x,y,p,\ep)\mapsto (x+\varepsilon\alpha,y+\varepsilon\beta,p+\varepsilon\gamma)
\end{equation}
is a relative contact transformation over $T_\ep$.
Moreover, all liftings of $\chi$ to $T_0$ are of the type $(\ref{AEPAL})$.
\end{enumerate}
\end{theorem}
\begin{proof}
See Theorems 2.4 and 2.6 of \cite{Mendes}.
\end{proof}

\section{Relative Legendrian Curves}\label{S:RLC}

 Let $\theta=\xi dx + \eta dy$ be the canonical $1$-form of $T^\ast \C^2 = \C^2 \times \C^2$. 
Hence $\pi=\pi_{\mathbb C^2} : \mathbb{P}^\ast \C^2=\C^2 \times \mathbb{P}^1 \to \C^2$ is given by $\pi(x,y; \xi : \eta)=(x,y)$. 
 Let $U\,[V]$ be the open subset of  $\mathbb{P}^\ast \C^2$ defined by $ \eta \neq 0\,[\xi \neq 0]$. 
 Then $\theta / \eta\,[\theta/\xi]$ defines a contact form $dy-pdx\,[dx-qdy]$ on $U\,[V]$, 
 where $p=- \xi / \eta\,[q=- \eta / \xi]$. 
 Moreover, $dy-pdx$ and $dx-qdy$ define the structure of contact manifold on $\mathbb{P}^\ast \C^2$.

 If $L$ is the germ of a Legendrian curve of $\mathbb P^*M$ and $L$ is not a fiber of $\pi_M$, $\pi_M(L)$ is the germ of a plane curve with irreducible tangent cone
 and $L=\mathbb P^*_{\pi_M(L)}M$.

Let $Y$ be the germ of a plane curve with irreducible tangent cone at a point $o$ of a surface $M$.
Let $L$ be the conormal of $Y$. Let $\sigma$ be the only point of $L$ such that $\pi_M(\sigma)=o$.
Let $k$ be the multiplicity of $Y$. Let $f$ be a defining function of $Y$. In this situation we will always choose a system of local coordinates $(x,y)$  of $M$ such that the tangent cone $C(Y)$ of $Y$ equals $\{y=0\}$.

 \begin{lemma}\label{GOODCURVE}
The following statements are equivalent$:$
\begin{enumerate}[$(a)$]
\item
mult$_\si(L)=$mult$_o(Y)$\em ; \em
\item
$C_\sigma(L)\not\supset (D\pi(\sigma))^{-1}(0,0)$\em ; \em
\item
$f\in (x^2,y)^k$\em ; \em
\item
if $t\mapsto (x(t),y(t))$ parametrizes a branch of $Y$, $x^2$ divides $y$.
\end{enumerate}
\end{lemma}
\begin{proof}
The equivalence of statements holds if and only if it holds for each branch.
Assume $Y$ irreducible. 
Assume $x(t)=t^k$ and $y(t)=t^n\varphi(t)=\widetilde\varphi(t)$, 
where $\varphi$ is a unit of $\mathbb C\{t\}$.
Since $C(Y)=\{y=0\}$, $n>k$. 
There is a unit $\psi$ of of $\mathbb C\{t\}$ such that $p(t)=t^{n-k}\psi(t)$.
Statements $(a)$ and $(b)$ hold if and only if $n-k\ge k$. 
Statement $(d)$ holds if and only if $n\ge 2k$.
Remark that 
$$
f=y^k+\textstyle{\sum_{i=1}^k}a_iy^{k-i}=\textstyle{\prod_{i=1}^k}(y-\widetilde\varphi(\theta^it))
$$
where $\theta=\exp(2\pi i/k)$.
Since $a_i$ is a homogeneous polynomial of degree $i$ on the $\widetilde\varphi(\theta^jt)$, $j=1,..,k$, $a_i\in (x^{[in/k]})$ and $a_k$ generates $(x^n)$. Therefore $(c)$ is verified if and only if $n/k\ge 2$.
\end{proof}

We say that a plane curve $Y$ is \emph{generic} 
[a Legendrian curve $L$ is in \emph{generic position}\/] if it verifies the conditions of Lemma \ref{GOODCURVE}.

Given a germ of a Legendrian curve $L$ of $U$ at $\sigma$ there is a germ of a contact transformation $\chi:(U,\sigma)\to (U,\sigma)$ such that $\chi(L)$ is in generic position (see \cite{KK} Corollary 1.6.4.). 

\begin{lemma}\label{VERYGOODCURVE}
Let $\sigma$ denote the origin of $U$. Assume $L,L_1,L_2$ are germs of Legendrian branches in generic position.
\begin{enumerate}[$(a)$]
\item 
$C_\sigma(L)=\{y=p=0\}$ if and only if given a parametrization
 $t\mapsto (x(t),y(t))$ of a branch of $Y$, $x^2\not\in (y)$.
\item 
$C_\sigma(L_1)\neq C_\sigma(L_2)$ if and only if $\pi(L_1)$ and $\pi(L_2)$ have contact of order $2$.
\end{enumerate}
\end{lemma}
\begin{proof}
Under the notations of Lemma \ref{GOODCURVE}, $C_\sigma(L)=\{y=p=0\}$ if $n\ge 2k+1$ and $C_\sigma(L)=\{y=p-\psi(0)x=0\}$ if $n=2k$.
\end{proof}

Remark that if $Y$ is a germ of a plane curve of $\mathbb C^2$ at the origin and $C(Y)=\{y=0\}$,
its conormal is a Legendrian variety contained in $U$.
By Darboux's Theorem each germ of a contact manifold of dimension $3$ is isomorphic 
to the germ of $U$ at $\sigma$, endowed with the contact structure of $U$ defined by $dy-pdx$.

\begin{definition}
Let $S$ be a reduced complex space. Let $Y$ be a reduced plane curve. 
Let $\mathcal Y$ be a deformation of $Y$ over $S$.
We say that $\mathcal Y$ is \em generic \em if its fibers are generic.
If $S$ is a non reduced complex space we say that $\mathcal Y$ is \em generic \em if $\mathcal Y$ admits a generic lifting.
\end{definition}

Given a flat deformation $\mathcal Y$ of a plane curve $Y$ over a complex space $S$ we will denote $\mathbb P_{\mathcal Y}^\ast (\C^2|S)$ by $\mathcal Con(\mathcal Y)$.

Consider the contact transformations from $\mathbb C^3$ to $\mathbb C^3$ given by
\begin{equation}\label{SEMIS}
\Phi(x,y,p)=(\lambda x,\lambda \mu y, \mu p),~~~ \lambda,\mu\in\mathbb C^*,
\end{equation}
\begin{equation}\label{PARABOLOIDAL}
\Phi(x,y,p)=(ax+bp,y+\frac{ac}{2}x^2+\frac{bd}{2}p^2+bcxp, cx+dp),
~~~\begin{vmatrix}
a & b\\
c & d
\end{vmatrix} 
=1,
\end{equation}
\begin{equation}\label{CHANGE}
\rho_\lambda(x,y,p)=(x,y-\lambda x^2/2, p-\lambda x), ~~~~~~~~ \lambda\in\mathbb C.
\end{equation}
The contact transformations (\ref{PARABOLOIDAL}) are called \em paraboloidal contact transformations. \em

\begin{example}\label{LIE}
\begin{enumerate}[$(a)$]
\item Let $k,n$ be integers such that $(k,n)=1$ and $0<k<n$. Let $Y=\{y^k-x^n=0\}$. Consider the contact transformation $\chi(x,y,p)=(p,y-xp,-x)$. The conormal $L$ of $Y$ is parametrized by
\[
x=t^k,\; y=t^n, \; p=\frac{n}{k}t^{n-k}.
\]
Therefore, $Y^\chi=\pi\left(\chi(L)\right)$ admits the equation $\left(xy/(k-n)\right)^k=x^{n-k}$.  We say that $Y^\chi$ is the action of the contact transformation $\chi$ on the plane curve $Y$.
\item Setting $Y=\{y^3-x^7=0\}$, $\chi(x,y,p)=(x+p,y+p^2/2,p)$, $Y^\chi$ admits a parametrization
\[
x=t^3+(7/3)t^4, \; y=t^7+(49/18)t^8.
\]
Changing parameters we get
\[
x=s^3, \; y=s^7+\lambda s^8+ h.o.t.,
\]
with $\lambda \neq 0$. Following \cite{Z}, $Y^\chi$ and $Y$ have the same topological type but are not analytically equivalent.
\end{enumerate}
\end{example}

\begin{theorem}\label{ALLCONTACT}\em (See \cite{AN} or \cite{Martins}.) \em
Let $\Phi:(\mathbb C^3,0)\to (\mathbb C^3,0)$ be the germ of a contact transformation. Then  $\Phi=\Phi_1\Phi_2\Phi_3$, where $\Phi_1$ is of  type $(\ref{SEMIS})$, $\Phi_2$ is of type $(\ref{PARABOLOIDAL})$ and $\Phi_3$ is of type $(\ref{abc})$, with 
$\alpha, \beta,\gamma\in\mathbb C\{x,y,p\}$. Moreover, there is $\beta_0\in \mathbb C\{x,y\}$ such that
$\beta$ verifies the Cauchy problem $(\ref{E:CAUCHY})$, $\beta-\beta_0\in (p)$   and 
\begin{equation}\label{CCOND}
\alpha,\beta,\gamma,\beta_0,\frac{\partial \alpha}{\partial x}, 
\frac{\partial \beta_0}{\partial x}, \frac{\partial \beta}{\partial p},
\frac{\partial^2 \beta}{\partial x\partial p} \in (x,y,p).
\end{equation}
If $D\Phi(0)(\{y=p=0\})=\{ y=p=0 \}$, $\Phi_2=id_{\mathbf C^3}$.
\end{theorem}

Let $\Sigma$ be an additive submonoid of the set of non negative integers.
Let $\Sigma_0$ be a minimal set of generators of $\Sigma$.
Let $\mathcal O_\Sigma$ be the set of power series $\sum_ia_it^i$ such that $a_i=0$ if $i\not\in\Sigma$. Let $\mathcal O^*_\Sigma$ be the set of power series $\sum_ia_it^i\in\mathcal O_\Sigma$ such that $a_i\not=0$ if $i\in\Sigma_0$.

\begin{lemma}\label{Wall} \em (Lemma 3.5.4 of \cite{WALL}) \em
Let $\alpha,\beta,\gamma\in\mathbb C\{t\}$. Assume $\alpha(0)\not=0$.

\noindent
$(a)$ If $(t\alpha)^k=t^k\gamma$, $\alpha\in \mathcal O_\Sigma$ if and only if $\gamma\in \mathcal O_\Sigma$ and $\alpha\in \mathcal O^*_\Sigma$ if and only if $\gamma\in \mathcal O^*_\Sigma$.

\noindent
$(b)$ If $t=s\beta(s)$ solves $s=t\alpha(t)$, $\alpha\in \mathcal O_\Sigma$ if and only if $\beta\in \mathcal O_\Sigma$ and $\alpha\in \mathcal O^*_\Sigma$ if and only if $\beta\in \mathcal O^*_\Sigma$.
\end{lemma}

\begin{theorem}[Theorem $1.3$, \cite{CN}]\label{L:EQUIEQUI}
Let $\chi: (\C^3,0) \to (\C^3,0)$ be a germ of  a contact transformation. Let $L$ be a germ of a Legendrian curve of $\C^3$ at the origin. If $L$ and $\chi(L)$ are in generic position, $\pi(L)$ and $\pi(\chi(L))$ are equisingular.
\end{theorem}
\begin{proof}

Assume $C_\sigma(L)$ is irreducible. 
Since when $\chi=\rho_\lambda$ or $\chi$ is of type (\ref{SEMIS})  
$\pi(L)$ and $\pi(\chi(L))$ are equisingular, we can assume that 
\[
C_\sigma(L)=C_\sigma(\chi (L))=\{y=p=0\}
\]
and $\chi$ is of type (\ref{abc}). Let $L_1, L_2$ be branches of $L$.
Let $S$[$k$] be the semigroup [multiplicity] of $\pi(L_1)$. Let $S'$ be the semigroup generated by $(S_0-k)\cap \mathbb N$. There are parametrizations
\begin{equation}\label{PARAMM}
t\mapsto (x_i(t),y_i(t),p_i(t))
\end{equation} 
of $L_i$, $i=1,2$ such that $x_1(t)=t^k$, $y_1\in \mathcal O_S^*$ and $p_1\in\mathcal O_{S'}$. 
By (\ref{CCOND}) $\chi(L_1)$ admits a parametrizaton (\ref{PARAMM}) with $x_1(t)=t^k\cdot$unit, $x_1\in\mathcal O_{S'}$, $y_1\in\mathcal O^*_{S}$. By Lemma \ref{Wall} we can assume that, after a reparametrization, $x_1(t)=t^k$ and $y_1\in\mathcal O^*_S$. Hence $\pi(L_1)$ and $\pi(\chi(L_1))$ are equisingular.

Assume $\pi(L_i)$ has multiplicity $k_i$, $i=1,2$ and $k$ is the least common multiple of $k_1,k_2$.  
Assume $\pi(L_1)$ and $\pi(L_2)$ have contact of order $\nu$. 
Then we can assume that $x_i(t)=t^{k_ik/k_j},~ \{i,j\}=\{1,2\}$,
\begin{equation}\label{CONTACT}
y_2\equiv y_1 ~\mod \mathcal O_S  ~ \hbox{ and } ~
y_2\not\equiv y_1 ~\mod \mathcal O_{S_+},
\end{equation}
where $S_\ell=\{0\}\cup \ell+\mathbb N$, $S=S_{\nu k}$, 
$S_+=S_{\nu k+1}$ and $S'=S_{\nu k-k}$.
Therefore $p_2\equiv p_1 \mod \mathcal O_{S'}$.
Composing $\chi$ with (\ref{PARAMM}) we obtain a parametrization
(\ref{PARAMM}) of $\chi(L_i)$ such that
\[
x_i=t^k\cdot\hbox{unit, } ~ 
x_2\equiv x_1 \mod \mathcal O_{S'} ~
\hbox{ and } 
y_2\equiv y_1 \mod \mathcal O_{S}, ~ i=1,2.
\]
By Lemma \ref{Wall}, after reparametrization, (\ref{CONTACT}) holds.
The theorem is proven when $C_\sigma(L)$ is irreducible.

Assume there is $\lambda_i$ such that $\pi(L_i)=\{y=\lambda_ix^2\}$, $i=1,2$ and $\lambda_1\not=\lambda_2$. If $\chi$ is paraboloidal, there are $\mu_i$ such that $\pi(\chi(L_i))=\{y=\mu_ix^2\}$, $i=1,2$ and $\mu_1\not=\mu_2$.
By Lemma \ref{VERYGOODCURVE} if $C_\sigma(L_1)\not =C_\sigma(L_2)$, the contact order of $\pi(L_1)$ and $\pi(L_2)$ equals $2$. Hence the truncation of the Puiseux expansion of  $\pi(L_i)$ equals $\lambda_ix^2$, $i=1,2$. Therefore 
the contact order of $\pi(\chi(L_1))$ and $\pi(\chi(L_2))$ equals $2$.
\end{proof}

\begin{definition}\label{EQUILEG}
Two Legendrian curves are \emph{equisingular} if their generic plane projections are equisingular. 
\end{definition}

\begin{lemma}
Assume $Y$ is a generic plane curve and $Y\hookrightarrow\mathcal Y$ defines an equisingular deformation of $Y$ with trivial normal cone along its trivial section. Then $\mathcal Y$ is generic.
\end{lemma}
\begin{proof}
By Proposition \ref{DECOMP} we can assume that $Y$ is irreducible.
Moreover, we can assume that $\mathcal Y$ is a deformation over a vector space and $C_{\{x=y=0\}}(\mathcal Y)=\{ y=0\}$. Let $x=t^k$, $y=t^n+\sum_{i\ge n+1}a_it^i$, $n\ge 2k$ be a parametrization of $Y$. After reparametrization, we can assume that $\mathcal Y$ admits a parametrization of the type
\begin{equation}\label{AJUDAP}
x=t^k, \qquad \textstyle y=\sum_i\alpha_it^i,
\end{equation}
with $\alpha_i\in\mathcal O_S$, $\alpha_i=0$ if $i<n$ and $k$ does note divide $i$. Since the normal cone of $\mathcal Y$ along its section is trivial, $\alpha_k=0$. Since (\ref{AJUDAP}) and
\[
\textstyle p=\sum_i i\alpha_it^{i-k}
\]
define a parametrization of $\mathcal Con(\mathcal Y)$, 
\[
C_{\{x=y=0\}}(\mathcal Con(\mathcal Y))=
\{ y=p-2k\alpha_{2k}x=0\}.
\]
\end{proof}

\begin{definition}\label{DEFLEGDEF}
Let $L$ be (the germ of) a Legendrian curve of $\mathbb C^3$ in generic position.
 Let $\mathcal L$ be a relative Legendrian curve over (a germ of) a complex space $S$ at $o$. We say that an immersion $i:L\hookrightarrow\mathcal L$ defines a \em deformation 
\begin{equation}\label{LEGMAP}
 \mathcal L \hookrightarrow \mathbb C^3\times S \to S
\end{equation}
of the Legendrian curve $L$ over $S$ \em if $i$ induces an isomorphism of $L$ onto $\mathcal L_o$ and there is a generic deformation $\mathcal Y$ of a plane curve $Y$ over $S$  such that
$\chi(\mathcal L)$ is isomorphic to $\mathcal Con \mathcal Y$ by a relative contact transformation verifying (\ref{DIDENTITY}). 

We say that the deformation (\ref{LEGMAP}) is \em equisingular \em if $\mathcal Y$ is equisingular. We denote by $\widehat{\mathcal{D}\textit{ef}}^{\;es}_{\;L}$ the category of equisingular deformations of $L$.
\end{definition}

\begin{remark}

We do not demand the flatness of the morphism (\ref{LEGMAP}). 

\end{remark}

\begin{lemma}
Using the notations of definition $\ref{DEFLEGDEF}$, given a section $\sigma:S\to \mathcal L$ of $\mathbb C^3\times S\to S$, there is a relative contact transformation $\chi$ such that $\chi\circ \sigma$ is trivial. Hence $\mathcal L$ is isomorphic to a deformation with trivial section.
\end{lemma}

\begin{proof}
We can assume that $S$ is the germ at the origin of a vector space. 
Set $\sigma(s)=(\overline{x}(s),\overline{y}(s),\overline{p}(s),s)$. 
Setting $\chi(x,y,p,s)=(x-\overline{x}(s),y-\overline{y}(s),p,s)$, we can assume that $\overline{x},\overline{y}$ vanish. Now 
$\chi(x,y,p,s)=(x,y-\overline{p}(s)x,p-\overline{p}(s),s)$ trivializes $\sigma$.
\end{proof}

\begin{theorem}\label{CONTACTEQUI}
Assume $\mathcal Y$ defines an equisingular deformation of a generic plane curve $Y$ with trivial normal cone along its trivial section. Let $\chi$ be a relative contact transformation verifying $(\ref{DIDENTITY})$. Then $\mathcal Y^\chi=\pi\left(\chi(\mathcal Con \mathcal Y)\right)$ is a generic equisingular deformation of $Y$.
\end{theorem}

\begin{proof}
We can assume that $S$ is the germ of a vector space. We only have to prove that $(i)$ $\mathcal (\mathcal Y^\chi)_s$ is generic and $(ii)$ $\mathcal (\mathcal Y^\chi)_s$ are equisingular, for small enough $s$. Let $\mathcal (\mathcal Y^\chi)_{s,i}$ be one branch of $\mathcal (\mathcal Y^\chi)_s$. Since $\mathcal (\mathcal Y^\chi)_{s,i}$ is generic its conormal admits a parametrization 
\[
\psi (t)=(t^k,t^n+h.o.t., (n/k)t^{n-k}+h.o.t.),
\]
with $n \geq 2k$ (see Lemma \ref{GOODCURVE}). By Theorem \ref{ALLCONTACT},  $\chi_s=\Phi_1\Phi_2\Phi_3$. Since $\Phi_1$ preserves genericity, we can assume $\Phi_1=id$. Notice that $\mathcal (\mathcal Y_{s,i})^{\Phi_2}$ is parametrized by
\begin{equation}\label{PPPA}
t \mapsto (x(t),y(t)),
\end{equation}
where $x(t)=at^k+b(n/k)t^{n-k}+h.o.t.$ and $y \in (t^{2k})$. If $s$ is small enough we can assume $a$ close to $1$ and $b$ close to $0$. Hence $(x)=(t^k)$. Therefore we can assume $\Phi_2=id$. Finally $\mathcal (\mathcal Y_{s,i})^{\Phi_3}$ is parametrized by (\ref{PPPA}), with
\[
x(t)=t^k+\psi^\ast(\alpha),\; y(t)=t^n+\psi^\ast(\beta).
\]
By $(\ref{CCOND})$ $(x)=(t^k)$ and $y \in (t^{2k})$ for small $s$. Now $(ii)$ follows from Theorem \ref{L:EQUIEQUI}, for $s$ small enough.
\end{proof}

\section{Deformations of the parametrization}\label{SECPARA}

Let $\psi: \bar{\mathbb C} \to \mathbb C^3$ be the parametrization  of a Legendrian curve $L$. We say that a deformation $\Psi$ of $\psi$ is a \emph{Legendrian deformation of $\psi$} if the analytic set parametrized by $\Psi$ is a relative Legendrian curve.
 We say that $(\chi,\xi)$ is an isomorphism of Legendrian deformations if $\chi : \C^3 \times T \to \C^3 \times T$ is a relative contact transformation (see (\ref{BIGDIAGRAM})).

\begin{definition}\label{DEFLEGDEF2}
Let $\varphi: \bar{\mathbb C} \to \mathbb C^2$ be the parametrization of a generic plane curve $Y$ with tangent cone $\{y=0\}$.
Let ${\mathcal{D}\textit{ef}^{\;es}_{\,\varphi}}$ be the category of equisingular deformations of $\varphi$. Let $\mathcal Y$ be an object of ${\mathcal{D}\textit{ef}^{\;es}_{\,\varphi}}$.
We say that $\mathcal Y$ is an object of the full subcategory 
$\overset{\twoheadrightarrow}{\mathcal{D}\textit{ef}^{\;es}_{\,\varphi}}$   of 
${\mathcal{D}\textit{ef}^{\;es}_{\,\varphi}}$ if $\mathcal Y$ is generic and
the normal cone of $\mathcal Y$ along $\{x=y=0\}$ equals $\{y=0\}$. 

Let $\psi: \bar{\mathbb C} \to \mathbb C^3$ be the parametrization of a curve $L$ in generic position.
We will denote by $\widehat{\mathcal{D}\textit{ef}}_{\,\psi}^{\;es}$ 
the category of equisingular  Legendrian  deformations of $\psi$.
\end{definition}

\begin{theorem}\label{FIXEDCONEGREUEL}
Let $\varphi: \bar{\mathbb C} \to \mathbb C^2$ be the parametrization of a generic plane curve $Y$ with tangent cone $\{y=0\}$. Then a semiuniversal deformation of $\varphi$ in ${\mathcal{D}\textit{ef}^{\;es}_{\,\varphi}}$ is also a semiuniversal deformation in $\overset{\twoheadrightarrow}{\mathcal{D}\textit{ef}^{\;es}_{\,\varphi}}$.
\end{theorem}
\begin{proof}
Assume $\varphi_i(t_i)=(x_i(t_i),y_i(t_i))$, $i=1,\ldots,r$. Let $I^{es}_\varphi$ be the vector space of the $a\partial_x+b\partial_y$ such that $a=[a_1,\ldots,a_r]^t$,  $b=[b_1,\ldots,b_r]^t$, where $a_i,b_i \in \C\{t_i\}t_i$ and 
\[
t_i \mapsto (x_i(t_i)+\ep a_i(t_i),y_i(t_i)+\ep b_i(t_i)),
\]
$i=1,\ldots,r$, is an equisingular deformation of $\varphi$ along the trivial section over $T_\ep$. Let $T_{\varphi}^{1,es}$ be the quotient of $I^{es}_\varphi$ by the linear subspace of its elements that define trivial deformations. Let 
\[
a^j\partial_x+b^j\partial_y, \qquad j=1,\ldots,\ell,
\]
be a family of representatives of a basis of $T_{\varphi}^{1,es}$. Set
\[
X_i=x_i+\textstyle \sum_{j=1}^\ell a_i^j s_j, \; Y_i=y_i+\textstyle \sum_{j=1}^\ell b_i^j s_j
\]
$i=1, \ldots,r$. By Theorem $2.38$ of \cite{GLS},
\[
\Phi_i(t_i)=(X_i(t_i), Y_i(t_i), \qquad i=1,\ldots,r,
\]
defines a semiuniversal deformation of $\varphi$ in  ${\mathcal{D}\textit{ef}^{\;es}_{\,\varphi}}$. It is enough to show that $\Phi_i$, $i=1, \ldots,r$ is an element of $\overset{\twoheadrightarrow}{\mathcal{D}\textit{ef}^{\;es}_{\,\varphi}}$. Let $m_i$ be the multiplicity of $\Phi_i$. Then $(x_i)=(t_i^{m_i})$. Since $\Phi_i$ is equimultiple $X_i, Y_i \in (t_i^{m_i})$. Since $y_i \in (t_i^{2m_i})$ and $\Phi_i$ is equisingular
\[
t_i \mapsto (X_i(t_i),Y_i(t_i)/X_i(t_i))
\]
is equimultiple (see \cite{GLS}). Therefore $Y_i \in (t_i^{2m_i})$.
\end{proof}

Assume $\psi$ is a parametrization of the conormal of the curve parametrized by $\varphi$. 
Let $\Phi[\Psi]$ be the deformation [Legendrian deformation] of $\varphi[\psi]$ given by
\[
\Phi_i(t_i,s)=(X_i(t_i,s),Y_i(t_i,s)),
\qquad
[\Psi_i(t_i,s)=(X_i(t_i,s),Y_i(t_i,s),P_i(t_i,s))].
\]
There are functors
$\mathcal{C}on:  \overset{\twoheadrightarrow}{\mathcal{D}\textit{ef}^{\;es}_{\,\varphi}} \to  \widehat{\mathcal{D}\textit{ef}}_{\,\psi}^{\;es}$,
$\pi:  \widehat{\mathcal{D}\textit{ef}}_{\,\psi}^{\;es} \to \overset{\twoheadrightarrow}{\mathcal{D}\textit{ef}^{\;es}_{\,\varphi}}$
given by 
\[
(\mathcal Con\Phi)_i=\left(X_i,Y_i,\frac{\partial Y_i}{\partial t}\left(\frac{\partial X_i}{\partial t}\right)^{-1} \right),
\qquad
(\Psi^\pi)_i=(X_i,Y_i).
\]

\begin{example}\label{EX}
Let $\Phi$ be the deformation $x=t^3$, $y=t^{10}+st^{11}$ 
of the plane curve $Y$ given by the equation $y^3-x^{10}$ 
and parametrized by $x=t^3$, $y=t^{10}$.
The deformation $\Phi$ induces the flat deformation given by 
\[
y^3-x^{10}-3sx^7y-s^3x^{11}.
\]
The conormal $\Psi$ of $\Phi$ is given by $x=t^3$, $y=t^{10}+st^{11}$, $3p=10t^{7}+11st^{8}$.

The  semigroup of the conormal curve of $\{y^3-x^{10}=0\}$ equals $\{3,6,7,9,10\} \cup \mathbb N +12$. The semigroup of the conormal of the deformed curve also contains the number $11$. Hence the deformation is not flat (see \cite{BG}).

It is shown in \cite{CN} that each flat deformation of the conormal of $y^k-x^n=0$ is rigid.
This result shows that the obvious choice of a definition of deformation of a Legendrian variety is not a very good one.
This is the reason to introduce Definitions \ref{DEFLEGDEF} and \ref{DEFLEGDEF2}.
\end{example}

\begin{definition}
Let $\mathcal{D}\textit{ef}^{\;es,\mu}_{\,\varphi}$ 
be the category given in the following way:
the objects of $\mathcal{D}\textit{ef}^{\;es,\mu}_{\,\varphi}$
are the objects of $\overset{\twoheadrightarrow}{\mathcal{D}\textit{ef}^{\;es}_{\,\varphi}}$;
the morphisms of $\mathcal{D}\textit{ef}^{\;es,\mu}_{\,\varphi}$
are the pairs $(\chi,\xi)$ where $\chi:\mathbb C^3\times T \to\mathbb C^3\times T $ is a relative contact transformation that acts on a deformation $\Phi$ by
\[
(\chi\cdot\Phi)_i=(\chi\circ\mathcal Con\Phi_i)^\pi,
\]
and leaves invariant the normal cone along $\{x=y=0\}$ of the image of $\Phi$.
\end{definition}

Notice that, by Theorem \ref{CONTACTEQUI} $\chi\cdot\Phi$ defined above is in fact an object of $\mathcal{D}\textit{ef}^{\;es,\mu}_{\,\varphi}$.

 Let $\mathfrak C_\varphi$ be a category of deformations of a curve parametrized by $\varphi$.
Let $S$ be a complex space. 
We will denote by $\mathfrak C_\varphi(S)$ the category of deformations of $\mathfrak C_\varphi$ over $S$.
We will denote by $\underline{\mathfrak C}_\varphi(S)$ the set of isomorphism classes of objects of 
$\mathfrak C_\varphi(S)$.

The functors 
$\mathcal{C}on:  \mathcal{D}\textit{ef}^{\;es,\mu}_{\,\varphi} \to  \widehat{\mathcal{D}\textit{ef}}_{\,\psi}^{\;es}$,
$\pi:  \widehat{\mathcal{D}\textit{ef}}_{\,\psi}^{\;es} \to \mathcal{D}\textit{ef}^{\;es,\mu}_{\,\varphi}$
are surjective and define natural equivalences between the functors
\[
T \mapsto \underline{\mathcal{D}\textit{ef}}^{\;es,\mu}_{\,\varphi}(T)
\qquad \hbox{and} \qquad
T \mapsto \widehat{\underline{\mathcal{D}\textit{ef}}}_{\,\psi}^{\;es}(T).
\]

\medskip

Let $\varphi:\overline{\mathbb C}\to \mathbb C^2$ be a parametrization of a generic plane curve $Y$ 
with irreducible components $Y_1,...,Y_r$.
Assume $\varphi_i(t)=(x_i(t_i),y_i(t_i))$, $i=1,...,r$. 

We will identify each ideal of $\mathcal{O}_Y$ with its image by $\vf^\ast : \mathcal{O}_Y \to \mathcal{O}_{\bar{\C}}$:
\[
\mathcal{O}_Y = \C \left\{ 
[x_1 \hdots x_r]^t
,
[y_1 \hdots y_r]^t
\right\}
 \subset \oplus_{i=1}^r \C\{t_i\}=\mathcal{O}_{\bar{\C}}.
\]	
Set $\dot{\mathbf{x}}=\left[ \dot{x}_1,\ldots,\dot{x}_r\right]^t$, where $\dot{x}_i$ is the derivative of $x_i$ with respect to $t_i$, $1\leq i \leq r$. Let
$\dot{\varphi} := \dot{\mathbf{x}} \partial_x +  \dot{\mathbf{y}}\partial_y$ 
be an element of the free $\mathcal{O}_{\bar{\C}}$-module
$\mathcal{O}_{\bar{\C}}{\partial}_x \oplus \mathcal{O}_{\bar{\C}}{\partial}_y$,
which has a structure of $\mathcal{O}_Y$-module induced by $\vf^\ast$.

Let $u_1,...,u_r,v_1,...,v_r\in\mathbb C\{t_i\}$. We say that
\[
(u_1,...,u_r)\partial_x\oplus (v_1,...,v_r)\partial_y
\]
belongs to the \emph{equisingularity module} $\Sigma_\varphi^{es}$ (see \cite{GLS}) of $\varphi$ if the deformation $\Phi$ given by 
$\Phi_i(t_i,\ep)=(x_i(t_i)+\ep u_i(t_i), y_i(t_i)+\ep v_i(t_i))$ is equisingular and has trivial normal cone along its trivial section.

Let $\mathfrak{m}_{\bar{\C}}\dot{\vf}$ be the sub $\mathcal{O}_{\bar{\C}}$-module of $\Sigma_\varphi^{es}$ generated by
\[
(a_1,\ldots,a_r)\left( \dot{\mathbf{x}}\partial_x +  \dot{\mathbf{y}}\partial_y\right),
\qquad 
a_i \in t_i\C\{t_i\},
\quad
1 \leq i \leq r.
\]
For $i=1,\ldots,r$ set $p_i=\dot{y}_i / \dot{x}_i$. For each $k \geq 0$ set
$\mathbf{p}^k = \left[ p_1^k,\ldots,p_r^k\right]^t$.
Let $ \widehat{I}$ be the sub $\mathcal{O}_Y$-module of 
$\mathcal{O}_{\bar{\C}}{\partial}_x \oplus \mathcal{O}_{\bar{\C}}{\partial}_y$  generated by
$(k+1)\mathbf{p}^k  \partial_x + k\mathbf{p}^{k+1}{\partial_y}$,  $k \geq 1$.

\begin{theorem} The module $ \widehat{I}$ is contained in $\Sigma_\varphi^{es}$ and
\[
\underline{\mathcal{D}\textit{ef}}^{\;es,\mu}_{\,\varphi}(T_\ep )
\simeq
\Sigma_\varphi^{es}/ 
(\mathfrak{m}_{\bar{\C}}\dot{\vf} + (x,y){\partial_x} \oplus (x^2,y){\partial_y} + \widehat{I}).
\]
\end{theorem}

\begin{proof}
Let $(u_1,...,u_r)\partial_x + (v_1,...,v_r)\partial_y \in  \widehat{I}$ and $\Phi$ be the deformation  given by 
\begin{equation}\label{E:DEF}
\Phi_i(t_i,\ep)=(x_i(t_i)+\ep u_i(t_i), y_i(t_i)+\ep v_i(t_i)).
\end{equation}
 We can suppose that for each $i=1,\ldots,r$
\[
u_i=p_i^\ell, \; v_i=\frac{\ell}{\ell+1}p_i^{\ell+1}
\]
for some $\ell \geq 1$. Because $Y$ is generic we have that $ord_{t_i} \, p_i > ord_{t_i} \, x_i$, $2ord_{t_i} \, p_i > ord_{t_i} \, y_i$ and, by Lemma  \ref{GOODCURVE}, $\Phi$ has generic fibres. The deformation $\Phi$ is the result of the action over the trivial deformation of $Y$ of the relative contact transformation
\[
\chi(x,y,p,\ep)=(x+\ep p^\ell,y+\ep \frac{\ell}{\ell+1}p^{\ell+1},p,\ep).
\]
As the trivial deformation is equisingular, $\Phi$ is equisingular.

Let $\Phi \in  \mathcal{D}\textit{ef}^{\;es,\mu}_{\,\varphi}$ be given as in (\ref{E:DEF}),
where $u_i,v_i \in \C\{t_i\},\, ord_{t_i}\,u_i \geq m_i,\, ord_{t_i} \, v_i \geq 2m_i,\, i=1, \ldots,r$, where $m_i$ is the multiplicity of $Y_i$. We have that $\Phi$ is trivial if and only if there are 
\begin{align*}
\xi_i(t_i)&=\widetilde{t}_i=t_i + \ep h_i,\\
\chi(x,y,p,\ep)&=(x+\ep \alpha,y + \ep \beta, p+ \ep \gamma,\ep),
\end{align*}
such that $\chi$ is a relative contact transformation, $\xi_i$ is an isomorphism, $\alpha,\beta, \gamma \in (x,y,p)\C\{x,y,p\},\, h_i \in t_i\C\{t_i\},\, 1 \leq i \leq r$, and
\begin{align*}
x_i(t_i) + \ep u_i(t_i)&=x_i(\widetilde{t}_i)+\ep \alpha(x_i(\widetilde{t}_i),y_i(\widetilde{t}_i),p_i(\widetilde{t}_i)),\\
y_i(t_i) + \ep v_i(t_i)&=y_i(\widetilde{t}_i)+\ep \beta(x_i(\widetilde{t}_i),y_i(\widetilde{t}_i),p_i(\widetilde{t}_i)),
\end{align*}
for $i=1,\ldots,r$. By Taylor's formula $x_i(\widetilde{t}_i)=x_i(t_i) + \ep \dot{x_i}(t_i)h_i(t_i), \, y_i(\widetilde{t}_i)=y_i(t_i) + \ep \dot{y_i}(t_i)h_i(t_i)$ and 
\begin{align*}
\ep \alpha(x_i(\widetilde{t}_i),y_i(\widetilde{t}_i),p_i(\widetilde{t}_i))&=\ep \alpha(x_i(t_i),y_i(t_i),p_i(t_i)), \\ 
\ep \beta(x_i(\widetilde{t}_i),y_i(\widetilde{t}_i),p_i(\widetilde{t}_i))&=\ep \beta(x_i(t_i),y_i(t_i),p_i(t_i)),
\end{align*}
for $i=1,\ldots,r$. Hence $\Phi$ is trivialized by $\chi$ if and only if
\begin{align}
u_i(t_i)&=\dot{x_i}(t_i)h_i(t_i) + \alpha(x_i(t_i),y_i(t_i),p_i(t_i)), \label{E:AAAA} \\
v_i(t_i)&=\dot{y}(t_i)h_i(t_i) + \beta(x_i(t_i),y_i(t_i),p_i(t_i)), \label{E:BBBB}
\end{align}
for $i=1,\ldots,r$. By Theorem~\ref{T:RCKT} (c), (\ref{E:AAAA}) and (\ref{E:BBBB}) are equivalent to the condition
\[
\mathbf{u}\partial_x +  \mathbf{v}\partial_y \in \mathfrak{m}_{\bar{\C}}\dot{\vf} + (x,y)\partial_x \oplus (x^2,y)\partial_y + \widehat{I}.
\]

\end{proof}

\begin{theorem} \label{T:DEFES}
Set $\ell=dim\, \underline{\mathcal{D}\textit{ef}}^{\;es,\mu}_{\,\varphi}(T_\ep )$.  Assume that
\begin{equation}\label{E:BASIS}
\mathbf{a}^j \frac{\partial}{\partial x} + \mathbf{b}^j \frac{\partial}{\partial y}= \left[
\begin{matrix} 
a_1^j\\ \vdots \\a_r^j
\end{matrix}
\right]\frac{\partial}{\partial x} +
\left[
\begin{matrix} 
b_1^j\\ \vdots \\b_r^j
\end{matrix}
\right]\frac{\partial}{\partial y},
\end{equation}
$ 1 \leq j \leq \ell$, represents a basis of $\underline{\mathcal{D}\textit{ef}}^{\;es,\mu}_{\,\varphi}(T_\ep )$. 
Let $\Phi: \bar{\C} \times\C^k \to \C^2 \times \C^k$ be the deformation of $\vf$ given by
\begin{equation}\label{E:DEFXYG}
 X_i(t_i,{\bf s})= x_i(t_i) + \sum_{j=1}^{\ell} a_i^j(t_i)s_j,\;
 Y_i(t_i,{\bf s})= y_i(t_i) + \sum_{j=1}^{\ell} b_i^j(t_i)s_j, 
\end{equation}
$i=1,\ldots,r$. Then $\mathcal{C}on \, \Phi$ is a semiuniversal deformation of $\psi$ in $\widehat{\mathcal{D}\textit{ef}}^{\; es}_{\, \psi}$.
\end{theorem}
This Theorem is the equivalent for Legendrian curves of Theorem $2.38$ of \cite{GLS} for plane curves.

\begin{remark}\label{R*}
Set 
\[
\overset{\twoheadrightarrow}{M}_\vf=\Sigma_\varphi^{es}/\left(\mathfrak{m}_{\bar{\C}}\dot{\vf} + (x,y)\partial_x \oplus (x^2,y)\partial_y \right).
\]
Then 
\[
 \overset{\twoheadrightarrow}{\underline{\mathcal{D}\textit{ef}}^{\;es}_{\,\vf}} (T_\ep) \cong \overset{\twoheadrightarrow}{M}_\vf.
\]
Let $k=dim\,  \overset{\twoheadrightarrow}{M}_\vf$ and assume that (\ref{E:BASIS}), 
$ 1 \leq j \leq k$, represents a basis of $\overset{\twoheadrightarrow}{M}_\vf$. 
Let $\Phi: \bar{\C} \times\C^k \to \C^2 \times \C^k$ be the deformation of $\vf$ given by
\begin{equation*}
 X_i(t_i,{\bf s})= x_i(t_i) + \sum_{j=1}^{k} a_i^j(t_i)s_j,\;
 Y_i(t_i,{\bf s})= y_i(t_i) + \sum_{j=1}^{k} b_i^j(t_i)s_j.
\end{equation*}
Then $\Phi$ is semiuniversal in $\overset{\twoheadrightarrow}{\mathcal{D}\textit{ef}^{\;es}_{\,\vf}}$ (see \cite{GLS} II Theorem 2.38). If $\Psi \in \widehat{\mathcal{D}\textit{ef}}^{\; es}_{\, \psi}(T)$, then $\Psi^{\pi} \in \overset{\twoheadrightarrow}{\mathcal{D}\textit{ef}^{\;es}_{\,\varphi}} (T)$. Hence there is $f:T \to \overset{\twoheadrightarrow}{M}_\vf$ such that $\Psi^\pi \cong f^\ast \Phi$. Therefore $\Psi=\mathcal{C}on \, \Psi^\pi \cong \mathcal{C}on \, f^\ast \Phi=f^\ast \mathcal{C}on \, \Phi$.
This shows that $\mathcal{C}on \, \Phi$ is  complete in  $\widehat{\mathcal{D}\textit{ef}}^{\; es}_{\, \psi}$. It is actually versal and the proof is only technically more complicated.
\end{remark}

\begin{proof}(of Theorem \ref{T:DEFES})
It is enough to show that $\mathcal{C}on \, \Phi$ is formally semiuniversal (see remark \ref{R*} and \cite{Flenner} Satz 5.2).

Let $\imath: T' \hookrightarrow T$ be a small extension. Let $\Psi \in \widehat{\mathcal{D}\textit{ef}}^{\; es}_{\, \psi}(T)$. Set $\Psi'= \imath^\ast \Psi$. Let $\eta': T' \to \C^\ell$ be a morphism of complex analytic spaces. Assume that $(\chi',\xi')$ define an isomorphism
\[
\eta'^\ast \mathcal{C}on \, \Phi \cong \Psi'.
\]
We need to find $\eta: T \to \C^\ell$ and $\chi,\xi$ such that $\eta' =\eta \circ \imath$ and $\chi,\xi$ define an isomorphism 
\[
\eta^\ast \mathcal{C}on \, \Phi \cong \Psi
\]
that extends $(\chi',\xi')$.

Let $A\,[A']$ be the local ring of $T \, [T']$. Let $\delta$ be the generator of $Ker(A \twoheadrightarrow A')$. We can assume $A' \cong \C\{\mathbf{z}\}/I$, where $\mathbf{z}=(z_1,\ldots,z_m)$. Set
\[
\widetilde{A}'=\C\{\mathbf{z}\} \quad \text{and} \quad \widetilde{A}=\C\{\mathbf{z},\ep\}/(\ep^2,\ep z_1,\ldots,\ep z_m).
\]
Let $\mathfrak{m}_A$ be the maximal ideal of $A$. Since $\mathfrak{m}_A \delta=0$ and $\delta \in \mathfrak{m}_A$, there is a morphism of local analytic algebras from $\widetilde{A}$ onto $A$ that takes $\ep$ into $\delta$ such that the diagram
\begin{equation}
\xymatrix{
\widetilde{A}   \ar[d] \ar[r] &\widetilde{A}'  \ar[d]\\
A   \ar[r] &A' 
}
\end{equation}
commutes. Assume $\widetilde{T} \,[\widetilde{T}']$ has local ring $\widetilde{A} \,[\widetilde{A}']$. We also denote by $\imath$ the morphism $\widetilde{T}' \hookrightarrow \widetilde{T}$. We denote by $\kappa$ the morphisms $T \hookrightarrow \widetilde{T} $ and $T' \hookrightarrow \widetilde{T}'$.  Let $\widetilde{\Psi} \in \widehat{\mathcal{D}\textit{ef}}^{\; es}_{\, \psi}(\widetilde{T})$ be a lifting of $\Psi$.

We fix a linear map $\sigma: A' \hookrightarrow \widetilde{A}'$ such that $\kappa^\ast \sigma=id_{A'}$. Set $\widetilde{\chi}'=\chi_{\sigma(\alpha), \sigma(\beta_0)}$, where  $\chi'=\chi_{\alpha,\beta_0}$. Define $\widetilde{\eta}'$ by  $\widetilde{\eta}'^\ast s_i=\sigma(\eta'^\ast s_i)$, $i=1,\ldots, l$. Let $\widetilde{\xi}'$ be the lifting of $\xi'$ determined by $\sigma$. Then
\[
\widetilde{\Psi}':= \widetilde{\chi}'^{-1} \circ \widetilde{\eta}'^\ast \mathcal{C}on \, \Phi \circ \widetilde{\xi}'^{-1}
\]
is a lifting of $\Psi'$ and
\begin{equation}\label{E:PSITIL'}
\widetilde{\chi}' \circ \widetilde{\Psi}' \circ \widetilde{\xi}'= \widetilde{\eta}'^\ast \mathcal{C}on \, \Phi.
\end{equation}

By Theorem~\ref{T:RCKT} it is enough to find liftings $\widetilde{\chi},\widetilde{\xi},\widetilde{\eta}$ of $\widetilde{\chi}',\widetilde{\xi}',\widetilde{\eta}'$ such that
\[
\widetilde{\chi} \cdot \widetilde{\Psi}^\pi \circ \widetilde{\xi} = \widetilde{\eta}^\ast \Phi
\]
in order to prove the theorem.\\

Consider the following commutative diagram
\[
\xymatrix{
\bar{\C} \times \widetilde{T}' \; \ar[d]^{\widetilde{\Psi}'}  \ar@{^{(}->}[r] &\bar{\C}\times \widetilde{T} \; \ar[d]^{\widetilde{\Psi}} \ar@{.>}[r] &\bar{\C}\times \C^\ell \ar[d]^{\mathcal{C}on \, \Phi}\\
\C^3 \times \widetilde{T}' \; \ar[d]^{pr} \ar@{^{(}->}[r] &\C^3 \times \widetilde{T} \; \ar[d]^{pr} \ar@{.>}[r] &\C^3 \times \C^\ell \ar[d]\\
\widetilde{T}' \; \ar@{^{(}->}[r] \ar@/_2pc/[rr]^{\widetilde{\eta}'} &\widetilde{T} \; \ar@{.>}[r]^{\widetilde{\eta}} &\C^\ell.
}
\]
If $\mathcal{C}on \, \Phi$ is given by 
\[
X_i(t_i,{\bf s}),\; Y_i(t_i,{\bf s}),\; P_i(t_i,{\bf s}) \; \in \C\{\mathbf{s},t_i\},
\]
then $\widetilde{\eta}'^{\ast} \, \mathcal{C}on \, \Phi$ is given by 
\[
X_i(t_i,\widetilde{\eta}'({\bf z})),\; Y_i(t_i,\widetilde{\eta}'({\bf z})),\; P_i(t_i,\widetilde{\eta}'({\bf z}))\; \in \widetilde{A}'\{t_i\}=\C\{\mathbf{z},t_i\}
\]
for $i=1,\ldots,r$.
Suppose that $\widetilde{\Psi}'$ is given by 
\[
U'_i(t_i,{\bf z}),\; V'_i(t_i,{\bf z}),\; W'_i(t_i,{\bf z})\; \in \C\{\mathbf{z},t_i\}.
\]
Then, $\widetilde{\Psi}$ must be given by
\[
U_i=U'_i + \ep u_i, \; V_i=V'_i + \ep v_i, \; W_i=W'_i + \ep w_i \; \in \widetilde{A}\{t_i\}= \C\{\mathbf{z},t_i\} \oplus \ep\C\{t_i\}
\]
with $u_i, v_i, w_i \in \C\{t_i\}$ and $i=1,\ldots,r$. By definition of deformation we have that, for each $i$,
\[
(U_i,V_i,W_i)=(x_i(t_i),y_i(t_i),p_i(t_i)) \; \mod \; \mathfrak{m}_{\widetilde{A}}.
\]
Suppose  $\widetilde{\eta}':\widetilde{T}' \to \C^\ell$ is given by $(\widetilde{\eta}'_1,\ldots,\widetilde{\eta}'_\ell)$, with $\widetilde{\eta}'_i \in \C\{{\bf z}\}$. Then $\widetilde{\eta}$ must be given by $\widetilde{\eta}=\widetilde{\eta}' + \ep\widetilde{\eta}^0$ for some $\widetilde{\eta}^0 = (\widetilde{\eta}^0_1,\ldots,\widetilde{\eta}^0_\ell) \in \C^\ell$. Suppose
that $\tilde{\chi}': \C^3\times \widetilde{T}' \to \C^3\times \widetilde{T'}$ is given at the ring level by
\[
(x,y,p) \mapsto (H'_1,H'_2,H'_3),
\]
such that $H'=id \; \mod \; \mathfrak{m}_{\widetilde{A}'}$ with $H'_i \in (x,y,p)A'\{x,y,p\}$. Let
the automorphism $\widetilde{\xi}': \bar{\C}\times \widetilde{T}' \to \bar{\C}\times \widetilde{T}'$ be given at the ring level by
\[
t_i \mapsto h'_i
\]such that $h'=id \; \mod \; \mathfrak{m}_{\widetilde{A}'}$ with $h'_i \in (t_i)\C\{\mathbf{z},t_i\}$.

Then, from ($\ref{E:PSITIL'}$) it follows that
\begin{align}\label{E3-PM}
\nonumber X_i(t_i,\widetilde{\eta}') &= H'_1(U'_i(h'_i),V'_i(h'_i),W'_i(h'_i)),\\
Y_i(t_i,\widetilde{\eta}') &= H'_2(U'_i(h'_i),V'_i(h'_i),W'_i(h'_i)),\\
\nonumber P_i(t_i,\widetilde{\eta}') &= H'_3(U'_i(h'_i),V'_i(h'_i),W'_i(h'_i)).
\end{align}
Now, $\widetilde{\eta}'$ must be extended to $\widetilde{\eta}$ such that the first two previous equations extend as well. That is, we must have
\begin{align}\label{E4-PM}
X_i(t_i,\widetilde{\eta}) &= (H'_1 + \ep\alpha)(U_i(h'_i+\ep h^0_i), V_i(h'_i+\ep h^0_i), W_i(h'_i+\ep h^0_i),\\
Y_i(t_i,\widetilde{\eta}) &= (H'_2 + \ep\beta)(U_i(h'_i+\ep h^0_i), V_i(h'_i+\ep h^0_i), W_i(h'_i+\ep h^0_i). \notag
 \end{align}
 with $\alpha,\beta \in (x,y,p)\C\{x,y,p\}$, $h^0_i \in (t_i)\C\{t_i\}$ such that 
 \[
(x,y,p) \mapsto (H'_1 + \ep \alpha,H'_2 + \ep \beta,H'_3 + \ep \gamma)
\]
gives a relative contact transformation over $ \widetilde{T}$ for some $\gamma \in (x,y,p)\C\{x,y,p\}$. The existence of this extended relative contact transformation is guaranteed by Theorem \ref{T:RCKT} (e). Moreover, this extension depends only on the choices of $\alpha$ and $\beta_0$.
 So, we need only to find $\alpha$, $\beta_0$, $\widetilde{\eta}^0$ and $h^0_i$ such that (\ref{E4-PM}) holds. Using Taylor's formula and $\ep^2=0$ we see that
\begin{align}\label{E5-PM}
\nonumber X_i(t_i, \widetilde{\eta}' + \ep \widetilde{\eta}^0) &= X_i(t_i,\widetilde{\eta}') + \ep \sum_{j=1}^{\ell}  \partial_{s_j} X_i (t_i, \widetilde{\eta}')\widetilde{\eta}^0_j\\
(\ep  \mathfrak{m}_{\widetilde{A}}=0) \qquad &= X_i(t_i,\widetilde{\eta}') + \ep \sum_{j=1}^{\ell} \partial_{s_j} X_i (t_i, 0)\widetilde{\eta}^0_j,\\
\nonumber Y_i(t_i, \widetilde{\eta}' + \ep \widetilde{\eta}^0) &= Y_i(t_i,\widetilde{\eta}') + \ep \sum_{j=1}^{\ell}  \partial_{s_j} Y_i (t_i, 0)\widetilde{\eta}^0_j.
\end{align}
Again by Taylor's formula and noticing that $\ep \mathfrak{m}_{\widetilde{A}}=0$, $\ep \mathfrak{m}_{\widetilde{A}'}=0$ in $\widetilde{A}$, $h'=id \; \mod \; \mathfrak{m}_{\widetilde{A}'}$ and $(U_i,V_i)=(x_i(t_i),y_i(t_i)) \; \mod \; \mathfrak{m}_{\widetilde{A}}$ we see that
\begin{align}\label{E6-PM}
\nonumber U_i(h'_i + \ep h^0_i) &= U_i(h'_i) + \ep \dot{U_i}(h'_i)h^0_i\\
&= U'_i(h'_i) + \ep (\dot{x_i}h_i^0 + u_i),\\
\nonumber V_i(h'_i + \ep h^0_i) &= V'_i(h'_i) + \ep (\dot{y_i}h_i^0 + v_i).
\end{align}
Now, $H' = id \; \mod \; \mathfrak{m}_{\widetilde{A}'}$, so
\[
\partial_x H'_1 = 1 \; \mod \; \mathfrak{m}_{\widetilde{A}'}, \qquad \partial_y H'_1, \partial_p H'_1 \in  \mathfrak{m}_{\widetilde{A}'} \widetilde{A}'\{x,y,p\}.
\]
In particular,
\[
\ep \partial_y H'_1= \ep \partial_p H'_1=0.
\]
By this and arguing as in (\ref{E5-PM}) and (\ref{E6-PM}) we see that
\begin{align*}
 \nonumber &(H'_1 + \ep\alpha)( U'_i(h'_i) + \ep (\dot{x_i}h_i^0 + u_i), V'_i(h'_i) + \ep (\dot{y_i}h_i^0 + v_i), W'_i(h'_i) + \ep (\dot{p_i}h_i^0 + w_i))\\
\nonumber &= H'_1(U'_i(h'_i),V'_i(h'_i),W'_i(h'_i)) + \ep(\alpha(U'_i(h'_i),V'_i(h'_i),W'_i(h'_i)) + 1(\dot{x_i}h_i^0 + u_i))\\
&= H'_1(U'_i(h'_i),V'_i(h'_i),W'_i(h'_i)) + \ep( \alpha (x_i,y_i,p_i) + \dot{x_i}h_i^0 + u_i),\\
&(H'_2 + \ep\beta)( U'_i(h'_i) + \ep (\dot{x_i}h_i^0 + u_i), V'_i(h'_i) + \ep (\dot{y_i}h_i^0 + v_i), W'_i(h'_i) + \ep (\dot{p_i}h_i^0 + w_i))\\
&= H'_2(U'_i(h'_i),V'_i(h'_i),W'_i(h'_i)) + \ep( \beta (x_i,y_i,p_i) + \dot{y_i}h_i^0 + v_i).
\end{align*}
Substituting this in (\ref{E4-PM}) and using (\ref{E3-PM}) and (\ref{E5-PM}) we see that we have to find $\widetilde{\eta}^0 = (\widetilde{\eta}^0_1,\ldots,\widetilde{\eta}^0_\ell) \in \C^\ell$, $h^0_i$ such that
\begin{align}\label{E7-PM}
&(u_i(t_i),v_i(t_i)) = \sum_{j=1}^{\ell} \widetilde{\eta}^0_j\left(\partial_{s_j} X_i (t_i, 0),\partial_{s_j} Y_i (t_i, 0)\right)-\\
 \nonumber - h^0_i(t_i)&((\dot {x_i}(t_i), \dot{y_i}(t_i)) - (\alpha(x_i(t_i),y_i(t_i),p_i(t_i)),\beta(x_i(t_i),y_i(t_i),p_i(t_i))).
\end{align}
Note that, because of Theorem \ref{T:RCKT} (c),  
\[
(\alpha(x_i(t_i),y_i(t_i),p_i(t_i)),\beta(x_i(t_i),y_i(t_i),p_i(t_i))) \in \widehat{I}
\]
 for each $i$. Also note that $\widetilde{\Psi} \in \widehat{\mathcal{D}\textit{ef}}^{\; es}_{\, \psi}(\widetilde{T})$ means that $(u_i,v_i) \in \Sigma_\varphi^{es}$.
Then, if the vectors
\begin{align*}
&\left(\partial_{s_j} X_1 (t_1, 0),\ldots,\partial_{s_j} X_r (t_r, 0)\right)\partial_x + \left(\partial_{s_j} Y_1 (t_1, 0),\ldots,\partial_{s_j} Y_r (t_r, 0)\right)\partial_y\\
&=(a_1^j(t_1),\ldots,a_r^j(t_r))\partial_x + (b_1^j(t_1),\ldots,b_r^j(t_r))\partial_y, \qquad j=1,\ldots,\ell
\end{align*}
 form a basis of  $\underline{\mathcal{D}\textit{ef}}^{\;es,\mu}_{\,\varphi}(T_\ep )$, we can solve (\ref{E7-PM}) with unique $\widetilde{\eta}^0_1, \ldots, \widetilde{\eta}^0_\ell$ for all $i=1,\ldots,r$. This implies that the conormal of $\Phi$ is a formally semiuniversal  equisingular deformation of $\psi$ over $\mathbb{C}^{\ell}$.
\end{proof}

\section{Deformations of the equation I}\label{SECI}

Let $Y$ be a generic curve with parametrization $\vf$ and equation $f$.
Let $L$ be the conormal of $Y$.

\begin{definition}
We will denote by
$\overset{\twoheadrightarrow}{\mathcal{D}\textit{ef}^{\;es}_{\,f}}$ (or $\overset{\twoheadrightarrow}{\mathcal{D}\textit{ef}^{\;es}_{\,Y}}$)
 the full subcategory of generic \emph{equisingular deformations}  of $($the equation $f$ of$)$ the plane curve $Y$ such that its normal cone along $\{x=y=0\}$ equals $\{y=0\}$. 
\end{definition}
Let $T$ be a complex space.
We associate to a deformation $\Phi$ of $\vf$ the deformation $\mathcal Y$
defined by the kernel of $\Phi^*:\mathcal O_{\mathbb C^2\times T}\to\mathcal O_{\overline C\times T}$.
We obtain in this way a functor
\[
\vartheta:\overset{\twoheadrightarrow}{\mathcal{D}\textit{ef}^{\;es}_{\,\vf}} \to \overset{\twoheadrightarrow}{\mathcal{D}\textit{ef}^{\;es}_{\,f}}.
\]
\begin{theorem}\label{ALPHA}
The functor $\vartheta$ is surjective and induces a natural equivalence between the functors
$T\mapsto \overset{\twoheadrightarrow}{\underline{\mathcal{D}\textit{ef}}^{\;es}_{\,\vf}}(T)$
and
$T\mapsto \overset{\twoheadrightarrow}{\underline{\mathcal{D}\textit{ef}}^{\;es}_{\,f}}(T)$.

Given a morphism of complex spaces $\sigma:T\to S$ and $\Phi\in \overset{\twoheadrightarrow}{\mathcal{D}\textit{ef}^{\;es}_{\,\vf}}(S)$,
\[
\sigma^*\vartheta(\Phi)=\vartheta(\sigma^*\Phi).
\]
\end{theorem}
\begin{proof}
See Theorem 2.64 of \cite{GLS}.
\end{proof}

Let $\mathcal Y$ be an object of $\overset{\twoheadrightarrow}{\mathcal{D}\textit{ef}^{\;es}_{\,\vf}}$.
Since the normal cone of $\mathcal Y$ along $\{x=y=0\}$ equals $\{y=0\}$,
$\mathcal Con(\mathcal Y)\subset U\times T$.

Let $\psi$ be the parametrization of the conormal of $\vf$.
Let $\Phi\in  \overset{\twoheadrightarrow}{\mathcal{D}\textit{ef}^{\;es}_{\,\vf}}(T)$.
Let $\Psi$ be the conormal of $\Phi$. 
Let $\widehat\vartheta(\Psi)$ denote the image of $\Psi$.
By Theorem \ref{T:RCONORMAL}
\begin{equation}\label{UPDOWN}
\widehat{\vartheta}(\Psi)=\mathcal Con(\vartheta(\Psi^\pi)).
\end{equation}

\begin{lemma}\label{ALPHAMU}
The functor $\widehat{\vartheta}$ is surjective and induces a natural equivalence between the functors
$T\mapsto\widehat{\underline{\mathcal{D}\textit{ef}}}_{\,\psi}^{\;es}(T)$
and
$T\mapsto\widehat{\underline{\mathcal{D}\textit{ef}}}_{\,L}^{\;es}(T)$.

Given a morphism of complex spaces $\sigma:T\to S$ and $\Psi\in\widehat{\mathcal \mathcal{D}\textit{ef}}_{\,\psi}^{\;es}(S)$,
\begin{equation}\label{SIGMASTAR}
\sigma^*\widehat{\vartheta}(\Psi)=\widehat{\vartheta}(\sigma^*\Psi).
\end{equation}
\end{lemma}
\begin{proof}
If $\mathcal L$ is in $\widehat{{\mathcal \mathcal{D}\textit{ef}}}_{\,L}^{\;es}(T)$, 
$\mathcal L^\pi$ is in $\overset{\twoheadrightarrow}{\mathcal{D}\textit{ef}^{\;es}_{\,f}}(T)$. 
Therefore $\mathcal L^\pi=\vartheta(\Phi)$, 
for some $\Phi\in\overset{\twoheadrightarrow}{\mathcal{D}\textit{ef}^{\;es}_{\,\vf}}(T)$.
Setting $\Psi=\mathcal Con(\Phi)$, $\widehat\vartheta(\Psi)=\mathcal L$.

By Theorem \ref{ALPHA} and (\ref{UPDOWN}), $\widehat\vartheta$ induces a natural equivalence and (\ref{SIGMASTAR}) holds.
\end{proof}

\begin{theorem}\label{PAREQ}
For each Legendrian curve $L$ there is a semiuniversal deformation $\mathcal L$ of $L$ in the category 
$\widehat{ \mathcal{D}\textit{ef}}^{\;es}_{\,L}$. Moreover, $\mathcal L$ is defined over a smooth analytic manifold.
\end{theorem}
\begin{proof}
Let $\Psi$ be the semiuniversal deformation of the parametrization $\psi$ of $L$ in the category
$\widehat{ \mathcal{D}\textit{ef}}^{\;es}_{\,\psi}$.
By Lemma \ref{ALPHAMU}, we can take $\mathcal L=\widehat \vartheta (\Psi)$.
\end{proof}

\section{Deformations of the equation II}\label{SECII}

\begin{definition}
Let $\mathcal{D}\textit{ef}^{\;es,\mu}_{\,f}$ (or $\mathcal{D}\textit{ef}^{\;es,\mu}_{\,Y}$)
be the category given in the following way:
the objects of $\mathcal{D}\textit{ef}^{\;es,\mu}_{\,f}$
are the objects of $\overset{\twoheadrightarrow}{\mathcal{D}\textit{ef}^{\;es}_{\,f}}$;
two objects $\mathcal Y,\mathcal Z$ of $\mathcal{D}\textit{ef}^{\;es,\mu}_{\,f}(T)$
are \em isomorphic \em  if
there is a relative contact transformation $\chi$ over $ T $ such that $\mathcal Z=\mathcal Y^\chi$.
\end{definition}

\begin{lemma}\label{CONDUTOR}
Assume $f\in\mathbb C\{x,y\}$ is the defining function of a generic plane curve $Y$.
Let $L$ be the conormal of $Y$.
For each $\ell\ge 1$ there is $h_\ell\in \mathbb C\{x,y\}$ such that
\[
(\ell+1)p^\ell f_x+\ell p^{\ell+1}f_y\equiv h_\ell  \mod I_L.
\]
Moreover, $h_\ell$ is unique modulo $I_Y$.
\end{lemma}

\begin{proof}
Let $\Delta$ be the germ of $\mathbb C$ at the origin.
Let $k_\tau $ $[c_\tau ]$ be the multiplicity [the conductor] of the branch $Y_\tau $ of $Y$, $\tau =1,...,n$.
Let $\sigma_\tau :\Delta\to L_\tau $ be the normalization of the conormal $L_\tau $ of $Y_\tau $, $\tau =1,...,n$.
Let $v_\tau $ be the valuation of $\mathbb C\{x,y,p\}$ associated to $\sigma_\tau $, $\tau =1,...,n$.
The restriction of $v_\tau $ to $\mathbb C\{x,y\}$ defines the valuation of 
$\mathbb C\{x,y\}$ associated to the normalization of $Y_\tau $, $\tau =1,...,n$.
By \cite{Z}, Section I.2
\begin{equation}\label{ESTIMATION}
v_\tau(f_{\tau, y})=c_\tau+k_\tau-1,
\qquad \hbox{and} \qquad
v_\tau(xf_{\tau ,x})=v_\tau(yf_{\tau, y}),
\end{equation}
for $\tau=1,...,n$.  By (\ref{ESTIMATION}) and  \cite{Z} there is $a_{\tau,\ell}\in \mathbb C\{x,y\}$ 
such that $v_\tau(\ell p^{\ell+1} f_{\tau, y}-a_{\tau,\ell})=+\infty$, $\tau=1,...,n$, for each $\ell\ge 1$.
Setting $a_{\ell}=\sum_{\tau=1}^na_{\tau,\ell}\prod_{j\neq \tau}f_j$,
\[
v_\tau(\ell p^{\ell+1} f_y-a_{\ell})=+\infty, \qquad \hbox{for  }\ell\ge 1, \quad \tau=1,...,n.
\]
A similar reasoning shows that there are $b_{\ell}\in \mathbb C\{x,y\}$ such that 
\[
v_\tau((\ell+1) p^{\ell} f_x-b_{\ell})=+\infty, \qquad \hbox{for  }\ell\ge 1, \quad \tau=1,...,n.
\]
\end{proof}

\begin{remark}\label{PRE-ACTION}
Assume $Y$ is irreducible with multiplicity $\nu$. Suppose $\mathcal Y \in \overset{\twoheadrightarrow}{\mathcal{D}\textit{ef}^{\;es}_{\,Y}}(T)$, where $T$ is a reduced complex space and let $\mathcal L$ be the relative conormal of $\mathcal Y$.  
Let $\Phi$ be the deformation of the parametrization of $Y$ such that $\vartheta(\Phi)=\mathcal Y$. Let $\Psi$ be the conormal of $\Phi$.
There $A_i\in\mathcal O_T$ such that
\[
\Psi^*x=t^\nu,
\quad
\Psi^*y=t^n+\sum_{i\ge n+1}A_it^i
\quad
\hbox{and}
\quad
\Psi^*p=\frac{n}{\nu}t^{n-\nu}+\sum_{i\ge n+1}\frac{i}{\nu}A_it^{i-\nu}.
\]
Given $f \in \mathcal{O}_T\{x,y,p\}$, $f \in I_{\mathcal L}$ if and only if $\Psi^*f=0$.
\end{remark}

\begin{theorem}\label{ACTION}
Let $Y$ be a generic curve.
Let $T$ be a complex space.
Let $\imath_0 :T\hookrightarrow T_0$ be a small extension and $\chi_0$ be a relative contact transformation over $T_0$.
Let $\mathcal Y_0\in \overset{\twoheadrightarrow}{\mathcal{D}\textit{ef}^{\;es}_{\,f}}(T_0)$, $\mathcal Y=\imath_0^*\mathcal Y_0$ and $\chi=\imath_0^*\chi_0$.
Assume $\chi_0$ equals $(\ref{AEPAL})$ and $\mathcal Y$ $[\mathcal Y_0,\mathcal Y^{\chi},\mathcal Y_0^{\chi_0}]$ are defined by $F$ $[F_0,F^{\chi},F_0^{\chi_0}]$, where $F_0=F+\ep g$, $g\in\mathbb C\{x,y\}$, and $F^{\chi}$ is a lifting of $f$. 
Then, if $F_0^{\chi_0}$ is a lifting of $F^\chi$,
\begin{equation}\label{AACCAO}
F_0^{\chi_0}=F^{\chi}+\ep g+\ep\alpha_0f_x+\ep \beta_0f_y+ 
\varepsilon {\textstyle\sum_{k\ge 1}}\frac{\alpha_k}{k+1}h_k.
\end{equation}
\end{theorem}

\begin{proof}
Remark that if $\chi$ equals (\ref{CHICHI}) and $I_{\mathcal Y}$ is generated by $F$, $I_{\mathcal Y^\chi}$ is generated by $F^\chi \in \mathcal O_{\C^2 \times S}$ such that
\begin{equation*}
F^\chi(x,y,s) \equiv F(x+A,y+B,s) \hbox{ mod } I_{\chi\left(\mathcal L\right)}.
\end{equation*}
Let $L$ denote the conormal of $Y$. Let $\mathcal L$$[\mathcal L_0]$ denote the relative conormal of $\mathcal Y$$[\mathcal Y_0]$.
We can assume $s=(s_1,...,s_m)$, 
\[
\mathcal O_T=\mathbb C\{s\},
~~ 
\mathcal O_{T_0}=\mathbb C\{s,\ep\}/\mathfrak n_\ep,
~~
\mathfrak n_\ep=(s_1\ep,...,s_m\ep,\ep^2).
\]
Since
$
I_{\chi_0\left(\mathcal L_0\right)} = I_{\chi\left(\mathcal L\right)} +  \ep\mathcal{O}_{\C^3\times T_0}\cap I_{\chi_0\left(\mathcal L_0\right)}=I_{\chi\left(\mathcal L\right)} +\ep I_L$ we have the following congruences modulo $I_{\chi_0\left(\mathcal L_0\right)}$:
\begin{align*}
F_0^{\chi_0} &\equiv F_0(x+A+\ep\alpha,y+B+\ep\beta,s,\ep)  \\
&\equiv F(x+A+\ep\alpha,y+B+\ep\beta,s) + \ep g  \\
&\equiv F(x+A,y+B,s) + \ep g+\ep\alpha\partial_x F+ \ep\beta\partial_y F  \\
&\equiv F^{\chi}+\ep g+\ep\alpha_0f_x+\ep \beta_0f_y+\varepsilon {\textstyle\sum_{k\ge 1}}\frac{\alpha_k}{k+1}h_k.
\end{align*}
\end{proof}

\begin{corollary}\label{CACTION}
Let $F=f+\varepsilon g$ be a defining function of a deformation 
$\mathcal Y\in \overset{\twoheadrightarrow}{\mathcal{D}\textit{ef}^{\;es}_{\,f}}(T_\varepsilon)$.
Let $\chi_{\alpha,\beta_0}$ be a contact transformation over $T_\varepsilon$.
Then 
\begin{equation}\label{ACCAO}
f+\varepsilon g+\varepsilon \alpha_0 f_x + \varepsilon \beta_0 f_y +
\varepsilon {\textstyle \sum_{k\ge 1}}\frac{\alpha_k}{k+1}h_k
\end{equation}
defines the action of $\chi_{\alpha,\beta_0}$ on $\mathcal Y$.
\end{corollary}

\begin{definition}
Let $f$ be a generic plane curve with tangent cone $\{y=0\}$. 
We will denote by $I_f$ the ideal of $\mathbb C\{x,y\}$ generated by the functions $g$ 
such that $f+\varepsilon g$ is equisingular over $T_\ep$ and has trivial normal cone along its trivial section.
We call $I_f$ the \emph{equisingularity ideal of} $f$.

We will denote by $I^\mu_f$ the ideal of $\mathbb C\{x,y\}$ generated by 
$f,(x,y)f_x$, $(x^2,y)f_y$ and $h_\ell$, $\ell\ge 1$.
\end{definition}

Let $f=\sum_{k,\ell}a_{k,\ell}$ be a convergent power series.
Let $u,v,d$ be positive integers. 
Assume $u,v$ coprime.
If $a_{k,\ell}\neq 0$ implies $uk+v\ell\ge d$ and there are $k_1,\ell_1,k_2,\ell_2$ such that 
$(k_1,\ell_1)\neq(k_2,\ell_2)$ and $a_{k_i,\ell_i}\neq 0$, $i=1,2$, we call
\[
\textstyle
f_{u,v,d}(x,y)=\sum_{uk+v\ell=d}a_{k,\ell}x^ky^\ell
\]
a \emph{face} of $f$.
We say that $f$ is \emph{semiquasihomogeneous} (\emph{SQH}) of type $(u,v;d)$ if $f_{u,v,d}$ is
a face of $f$ and $f_{u,v,d}$ has isolated singularities. We say that $f$ is \emph{Newton non-degenerate} (\emph{NND})
if $x,y$ do not divide $f$ and the singular locus of each face of $f$ is contained in $\{xy=0\}$.

\begin{lemma}\label{HLEMMA}
If $f$ is generic, $I^\mu_f\subset I_f$. 
\end{lemma}

\begin{proof}
Let $\alpha\in (x,y)$, $\beta\in (x^2,y)$. Set $\chi=\chi_{\alpha,0}~[\chi=\chi_{0,\beta},\chi=\chi_{p^\ell,0}]$. 
By Lemma \ref{ACTION}, $f^\chi$ equals
\[
f+\ep\alpha f_x, 
\qquad
[f+\ep\beta f_y, 
\quad
f+\ep h_\ell/(\ell+1)].
\]
By Lemma \ref{CONTACTEQUI}, $f^\chi$ is equisingular. 
Since the derivative of $\chi$ leaves invariant $\{y=0\}$, then $(x,y) f_x$, $(x^2,y)f_y\subset I_f$ and $h_\ell\in I_f$, for each $\ell\ge 1$.
\end{proof}

\begin{theorem}\label{Tep}
If $f$ is generic, 
\[\underline{\mathcal{D}\textit{ef}}^{\;es,\mu}_{\,f}(T_\ep)\simeq I_f/I^\mu_f.\]
\end{theorem}

\begin{proof}
Let $G\in \mathcal{D}\textit{ef}^{\;es,\mu}_{\,f}(T_\ep)$. 
There is $g\in I_f$ such that $G=f+\ep g $.
The deformation $f+\ep g $ is trivial in $\mathcal{D}\textit{ef}^{\;es,\mu}_{\,f}(T_\ep)$  
if and only if there are $h \in \C\{x,y\}$ and a contact transformation  (\ref{abc})
 such that
\begin{equation}\label{CONDICAO2}
G(x+\alpha,y+\beta,\ep)=(1+\ep h)f 
\qquad  \hbox{ mod }  \varepsilon I_{L}. 
\end{equation}  
By Corollary \ref{CACTION}, (\ref{CONDICAO2}) holds if and only if
\[
g+\alpha_0 f_x + \beta_0 f_y
+\sum_{\ell } \frac{\alpha _\ell}{\ell+1} h_\ell
=h f\hbox{ mod } (f).
\]
Hence $G$ is trivial if and only if $g \in I_f^\mu$.
\end{proof}

\begin{remark}\label{PRE-LAST}
 Each equisingular deformation $F$ of a SQH or NND plane curve $f$ is isomorphic to a deformation $\widetilde{F}$, such that $\widetilde{F}$ is equisingular via trivial sections (see \cite{WAHL} and \cite{GLS}). This means that, in the SQH or NND case, if $A \twoheadrightarrow A'$ is a small extension with kernel $\ep$ such that $\Y' \in  \mathcal{D}\textit{ef}^{\;es,\mu}_{\,f}(A'), \Y \in \mathcal{D}\textit{ef}^{\;es,\mu}_{\,f}(A)$ defined by $F'$, respectively $F=F'+\ep a(x,y)$, then $f+\ep a(x,y)$ defines a deformation in $\mathcal{D}\textit{ef}^{\;es,\mu}_{\,f}(T_\ep)$(see Theorem 8.2 of \cite{WAHL}). 
\end{remark}

\begin{theorem}\label{LAST}
Assume $Y$ is a generic plane curve with conormal $L$, defined by a  power series $f$.
Assume $f$ is SQH or $f$ is NND.
If $g_1,...,g_n\in I_f$ represent a basis of $I_f/I^\mu_f$ with Newton order $\geq 1$, 
the  deformation $\mathcal G$ defined by
\begin{equation}\label{VERSALEQUATION}
G(x,y,s_1,...,s_n)=f(x,y)+\sum_{i=1}^n s_ig_i
\end{equation}
is a semiuniversal deformation of  $f$ in $\mathcal{D}\textit{ef}^{\;es,\mu}_{\,f}$.
\end{theorem}

\begin{proof} 
The choice of $g_1,...,g_n$ identifies $I_f/I^\mu_f$ with $\mathbb C^n$.
It is enough to show that (\ref{VERSALEQUATION}) 
is a formally versal deformation of $f$
in $\mathcal{D}\textit{ef}^{\;es,\mu}_{\,f}$ and there is a versal deformation of $f$ in $\mathcal{D}\textit{ef}^{\;es,\mu}_{\,f}$ (see \cite{Flenner} Satz $5.2$).
The second requirement follows from Theorem \ref{PAREQ}. 
Let us prove that the first requirement is fulfilled.
We will follow the terminology of the proof of Theorem \ref{ACTION}. 
Let $\eta:T\to\mathbb C^n$ be a morphism of complex spaces and let
$\chi$ be a relative contact transformation over $T$ such that $\eta^*\mathcal G=\mathcal Y^\chi$.
It is enough to show that there is a unique pair $(\eta_0,\chi_0)$ where $\eta_0$ is a morphism from $T_0$ to $\mathbb C^n$ and $\chi_0$ is a relative contact 
transformation  over $T_0$ such that
\begin{equation}\label{REQ}
\eta_0\circ\imath_0=\eta
\qquad
\hbox{and}
\qquad
\eta_0^*\mathcal G=\mathcal Y_0^{\chi_0}.
\end{equation}
Because $\eta^*\mathcal G=\mathcal Y^\chi$ there is $h\in (s)\mathcal O_{\mathbb C^2\times T}$ such that
\[
(1+h)\eta^*G=F^\chi.
\]
In order for \ref{REQ} to hold, we need to find $a\in\mathbb C^n$
 , $\sigma\in\mathcal O_{\mathbb C^2}$ and $\chi_0$ such that
\[
\eta^0=\eta+\ep a,
\qquad
\hbox{and}
\qquad
(1+h+\ep \si)\eta_0^*G=F_0^{\chi_0}.
\]
By Theorem \ref{T:RCKT} there are $A,B_0$ such that 
\[
\chi=\chi_{A,B_0}
\]
and $\chi_0$ exists if and only if there are $\alpha,\beta_0$ such that
\[
\chi_0=\chi_{A+\ep \alpha,B_0+\ep\beta_0}.
\]
By Theorem \ref{ACTION}, $F_0^{\chi_0}$ equals (\ref{AACCAO}).
Moreover,
\begin{equation}
\begin{split}
(1+h+\ep\sigma)\eta^*_0G & = (1+h)\eta^*G +\ep\sigma\eta^*G+\ep (1+h){\textstyle\sum_{i=1}^na_ig_i} \\
 & = F^\chi  +\ep \sigma f+\ep (1+h) {\textstyle\sum_{i=1}^n a_ig_i}.
\end{split}
\end{equation}
Hence we need to solve the equation
\begin{equation}\label{REL}
g(1+h)^{-1}= {\textstyle\sum_{i=1}^na_ig_i}-(1+h)^{-1}(\ep\sigma f+\alpha_0f_x + \beta_0f_y +
 {\textstyle \sum_\ell\frac{\alpha_\ell}{\ell+1}h_\ell).}
\end{equation}
Since, as noted in Remark \ref{PRE-LAST}, $g(1+h)^{-1}\in I_f$ there are unique $a_1,...,a_n$ such that
\[
g(1+h)^{-1}-{\textstyle\sum_{i=1}^na_ig_i}\in I^\mu_f.
\]
Hence there are $\alpha_\ell,\beta_0,\sigma$ such that (\ref{REL}) holds.
\end{proof}

\begin{corollary}
The relative conormal of $\mathcal G$ is a semiuniversal deformation of the conormal $L$ of $Y$ on
$\widehat{\mathcal{D}\textit{ef}}^{\;es}_{\,L}$.
\end{corollary}

\begin{proof}
Suppose $\imath: T' \hookrightarrow T$ is an embedding of complex spaces, $\mathcal L \in \widehat{\mathcal{D}\textit{ef}}^{\;es}_{\,L}(T)$, $\mathcal L'=\imath^\ast \mathcal L \in \widehat{\mathcal{D}\textit{ef}}^{\;es}_{\,L}(T')$. Let $\eta':T' \to \C^n$ be a morphism of complex spaces and $\chi'$ a relative contact transformation such that
\begin{equation}\label{E:PRE-LAST1}
\chi'( \EL') = \eta'^\ast \mathcal Con( \mathcal G).
\end{equation}
Let $\Y'=\pi(\EL')$ and $\Y=\pi(\EL)$. Equation (\ref{E:PRE-LAST1}) implies that ${\Y'}^{\chi'}= \eta'^\ast \GE \in \mathcal{D}\textit{ef}^{\;es,\mu}_{\,f}(T')$. Because $\GE$ is semiuniversal, there is $\eta:T \to \C^n$ with $\eta'= \eta \circ \imath$ and $\chi$ relative contact transformation extending $\chi'$ such that $\Y^\chi=\eta^\ast \GE$. This means that $\eta^\ast \mathcal{C}on(\GE)= \chi(\EL)$, hence $\mathcal Con(\GE)$ is  semiuniversal.
\end{proof}

\begin{figure}[h]
\centering
\includegraphics[ trim={0 9cm 0 0},clip, scale=1]{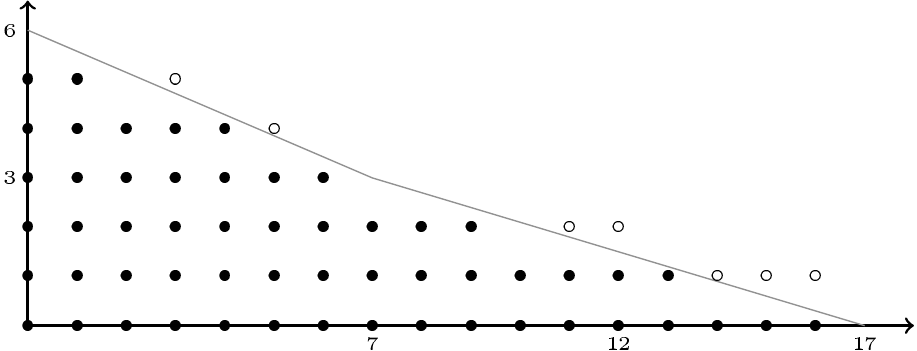}
\caption{Monomial base for  $\frac{\C\{x,y\}}{\left( f, (x,y)f_x,(x^2,y)f_y\right)}$. } \label{fig:M3}
\end{figure}

\begin{example}
If $f(x,y)=(y^3+x^7)(y^3+x^{10})$, $f$ is NND and $I_f$ is generated by the polynomials $x^2f_y,yf_x$ and $x^iy^j$ such that
$3i+7j\ge 42$ and
$3i+10j\ge 51$ (see Proposition $2.17$ of \cite{GLS}).

A semiuniversal object in $\overset{\twoheadrightarrow}{\mathcal{D}\textit{ef}^{\;es}_{\,f}}$ (see Proposition 2.69 and Corollary 2.71 of \cite{GLS}) is given by:
\[
f(x,y)+s_1x^3y^5+s_2x^5y^4+s_3x^{11}y^2+s_4x^{12}y^2+s_5x^{14}y+s_6x^{15}y+s_7x^{16}y.
\]
See fig. \ref{fig:M3}. According to Theorem \ref{LAST}, the deformation defined by
\[
f(x,y)+s_1x^3y^5+s_2x^{5}y^4+s_3x^{14}y
\]
is a semiuniversal deformation of  $f$ in $\mathcal{D}\textit{ef}^{\;es,\mu}_{\,f}$.

\end{example}

\end{document}